\documentclass{amsart}
\usepackage{amssymb}
\usepackage{verbatim}
\usepackage{array}
\usepackage{latexsym}
\usepackage{enumerate}
\usepackage{amsmath,mathtools}
\usepackage{amsfonts}
\usepackage{amsthm}
\usepackage[foot]{amsaddr}
\usepackage{color}
\usepackage[english]{babel}
\usepackage{graphicx}
\usepackage{hyperref}
\usepackage[capitalise]{cleveref}

\newtheorem{theorem}{Theorem}[section]
\newtheorem{proposition}[theorem]{Proposition}
\newtheorem{lemma}[theorem]{Lemma}
\newtheorem{corollary}[theorem]{Corollary}
\newtheorem{definition}[theorem]{Definition}
\newtheorem{observation}[theorem]{Observation}

\newtheorem{remark}[theorem]{Remark}

\def\F{\mathbf F}

\def\Z{\mathbf Z}

\def\deg{\mbox{\rm deg}}

\def\min{{\rm min}}

\newcommand{\PGU}{\mbox{\rm PGU}}

\newcommand{\aut}{\mbox{\rm Aut}}

\newcommand{\al}{\alpha}

\newcommand{\PP}{\mathcal{P}}
\newcommand{\QQ}{\mathcal{Q}}
\newcommand{\RR}{\mathcal{R}}

\newcommand{\elltilde}{K}
\newcommand{\cc}{\mathfrak{c}}
\newcommand{\dd}{\mathfrak{d}}

\renewcommand{\PP}{\mathcal{P}}
\renewcommand{\QQ}{\mathcal{Q}}
\renewcommand{\F}{\mathbb{F}}
\renewcommand{\Z}{\mathbb{Z}}

\renewcommand{\PGU}{\mathrm{PGU}}
\renewcommand{\aut}{\mathrm{Aut}}

\newcommand{\Yte}{\overline{Y}_3}
\newcommand{\He}{\overline{H}_1}

\title[On a maximal function field with the third largest possible genus, $q \equiv 1 \pmod 3$]{Weierstrass semigroups and automorphism group of a maximal function field with the third largest possible genus, $q \equiv 1 \pmod 3$}
\date{}

\author[Peter Beelen]{Peter Beelen$^1$} \address{$^1$Department of Applied Mathematics and Computer Science, Technical University of Denmark, Kongens Lyngby 2800, Denmark} \email{pabe@dtu.dk} \thanks{}
\author[Maria Montanucci]{Maria Montanucci$^1$} \address{} \email{marimo@dtu.dk} \thanks{}
\author[Lara Vicino]{Lara Vicino$^2$} \address{$^2$Faculty of Science and Engineering - Bernoulli Institute, University of Groningen, Groningen 9747AG, The Netherlands}  \email{l.vicino@rug.nl} \thanks{}

\begin{document}

\begin{abstract}
In this article we continue the work started in \cite{BMV}, explicitly determining the Weierstrass semigroup at any place and the full automorphism group of a known $\mathbb{F}_{q^2}$-maximal function field $Y_3$ having the third largest genus, for $q \equiv 1 \pmod 3$. This function field arises as a Galois subfield of the Hermitian function field, and its uniqueness (with respect to the value of its genus) is a well-known open problem. Knowing the Weierstrass semigroups may provide a key towards solving this problem. Surprisingly enough, $Y_3$ has many different types of Weierstrass semigroups and the set of its Weierstrass places is much richer than its set of $\mathbb{F}_{q^2}$-rational places. We show that a similar exceptional behaviour does not occur in terms of automorphisms, that is, $\aut(Y_3)$ is exactly the automorphism group inherited from the Hermitian function field, apart from small values of $q$.
\end{abstract}

\maketitle



\section{Introduction}

A function field $F$ over a finite field with square cardinality is called maximal if the Hasse-Weil upper bound is attained. More precisely, if $F$ is a function field of genus $g$ over the finite field $\mathbb{F}_{q^2}$ with $q^2$ elements, then the Hasse-Weil upper bound states that
$$N(F)= q^2 + 1 + 2qg,$$
where $N(F)$ denotes the number of places of degree $1$ (we will refer to them also as $\mathbb{F}_{q^2}$-rational places) of $F$. 

$\mathbb{F}_{q^2}$-maximal function fields, especially those with large genus, have been and currently are investigated also in connection with coding theory and cryptography, due to Goppa's construction of error-correcting codes. This raises the natural question: how large can the genus of a maximal function field be? How many maximal function fields exist for a fixed genus $g$ up to isomorphism?

It is well known that $g \leq q(q - 1)/2$, see \cite{Y}, and that $g$ reaches this upper limit if and only if $F$ is isomorphic to the Hermitian function field $H_1:=\mathbb{F}_{q^2}(u,v)$ with $v^{q+1}=u^q+u$, see \cite{RS}. In \cite{FT} it is proven that either $g \leq \lfloor (q- 1)^2/4\rfloor$, or $g = q(q- 1)/2$.
For $q$ odd, $g = (q-1)^2/4$ occurs if and only if $F$ is isomorphic to the function field $H_2:=\mathbb{F}_{q^2}(x,y)$ with $y^q + y = x^{(q+1)/2}$, see \cite[Thm. 3.1]{FGT}. 
For $q$ even, a similar result is obtained in \cite{AT}: $g = \lfloor(q -1)^2/4\rfloor = q(q -2)/4$ occurs if and only if $F$ is isomorphic to
the the function field $H_2:=\mathbb{F}_{q^2}(x,y)$ with $y^{q/2} + \ldots + y^2 + y = x^{q+1}.$ 

Even though the first $g_1:=q(q-1)/2$ and the second $g_2:=\lfloor (q-1)^2/4 \rfloor$ largest genera of $\mathbb{F}_{q^2}$-maximal function fields are known, and they are realized by exactly one function field up to isomorphism, the same is not clear for the third largest genus. Its value is known to be equal to $g_3:=\lfloor (q^2-q+4)/6 \rfloor$, but it is still unclear whether this is realized by exactly one function field up to isomorphism or not, see \cite{KT}.

In fact, it is shown in \cite[Remark 3.4]{KT} that an $\mathbb{F}_{q^2}$-maximal function field of genus $\lfloor (q^2-q+4)/6 \rfloor$ does exist, namely
\begin{enumerate}
\item $X_3:=\mathbb{F}_{q^2}(x,y)$ with $x^{(q+1)/3} + x^{2(q+1)/3} + y^{q+1} = 0$, if $q \equiv 2 \pmod 3$;

\item $Y_3:=\mathbb{F}_{q^2}(x,y)$ with $x^q+xy^{(q-1)/3}-y^{(2q+1)/3}=0$, if $q \equiv 1 \pmod 3$; and

\item $Z_3:=\mathbb{F}_{q^2}(x,y)$ with $y^q + y + (\sum_{i=1}^{t} x^{q/3^i}
)^2 = 0$, if $q = 3^t$.
\end{enumerate}
All the examples above are Galois subfields of the Hermitian function field $H_1$ (where the Galois group has order three). 
Understanding whether or not these function fields are the only $\mathbb{F}_{q^2}$-maximal function fields of genus $g_3$ is a well-known open problem.

In the proof of uniqueness for the cases $g_1$ and $g_2$ from \cite{RS} and \cite{FT}, the so-called Weiestrass semigroups and Weierstrass places play a central role, so it is natural to ask how these behave for the three function fields mentioned above.

Given a place $P$ on $F$, the Weierstrass semigroup $H(P)$ is defined as the set of non-negative integers $n$ for which there exist a function $f \in F$ with $(f)_\infty=nP$.
From the Weierstrass gap Theorem \cite[Theorem 1.6.8]{Sti}, the set $G(P) :=\mathbb{N} \setminus H(P)$ contains exactly $g$ elements (called gaps). Clearly the set $H(P)$ in general may vary as the place $P$ varies, and different Weierstrass semigroups can be expected at different places of $F$. However, it is known that generically the semigroup $H(P)$ is the same, but that there can exist finitely many places of $F$, called Weierstrass places, with a different set of gaps.

The interest and potential of these objects in relation with finding characterizing properties of function fields was in fact already pointed out in St\"ohr-Voloch theory \cite{SV}. Apart from the two characterizations for $g_1$ and $g_2$ mentioned above, important characterization results using Weierstrass places, semigroups and automorphism groups can be found in \cite{FGT}, \cite{FT1} (for the Suzuki function field) and \cite{TT} (for the Ree function field).

In order to provide similar tools for maximal curves of third largest genus, it is natural to wonder whether Weierstrass semigroups, Weierstrass places and automorphism group can be completely determined for maximal function fields of genus $g_3$. This analysis in the case of $X_3$ (that is, $q \equiv 2 \pmod 3$) has been already carried out in \cite{BMV}.
Continuing our investigation, in this paper we compute the Weierstrass semigroup at every place of the function field $Y_3$ having third largest genus $g_3$ for $q \equiv 1\pmod 3$, as well as its set of Weierstrass places and its full automorphism group.
Doing so, we show that $Y_3$ has a quite large set of non-rational Weierstrass places and many different types of Weierstrass semigroups. The full automorphism group of $Y_3$ is also computed as an application of the results mentioned above.

The paper is organized as follows. Section 2 provides the necessary preliminary results on algebraic function fields, Weierstrass semigroups and the function field $Y_3$. In the same section, we describe some initial properties of $Y_3$ as well as a first analysis of some Weierstrass semigroups. These initial computations are then used in Section 3 to determine the automorphism group $\aut(Y_3)$. In Section 4, we define the $\mathcal{P}$-order of a place $P$, and the construction of the main functions used to compute $H(P)$, in analogy to those constructed in \cite{BMV}. 
These functions are then used in Sections 5 and 6, which contain the proofs of the main theorems of the paper, namely the description of the Weierstrass semigroup at rational and non-rational places of $Y_3$, respectively. 

\section{Preliminary results}

In this section, we provide some preliminary results that will be used throughout the paper. In the first subsection, the function field $Y_3$ is presented and some of its useful properties are pointed out. There some principal divisors are computed, as well as some differentials and the corresponding canonical divisors on $Y_3$. 
Local power series expansions for some useful functions can also be found there. 
In the second subsection we will instead start our investigation on Weierstrass semigroups of $Y_3$ by computing $H(P)$ at some special places $P$.

\subsection{First properties of the function field \texorpdfstring{$Y_3$}{Y\_3}}

In this subsection, we establish various facts concerning the function field $Y_3$ that will be useful later on. As a byproduct of these, we will be able to compute the full automorphism group of $Y_3.$ We first revisit results from \cite[Example 6.3]{GSX}. For convenience, we define $m:=(q-1)/3$.

One first considers the function field ${Y}_3$ as an index three subfield of the Hermitian function field ${H}_1$. More precisely, consider the Hermitian function field $H_1:=\mathbb{F}_{q^2}(u,v)$ with $v^{q+1}=u^q+u$ and define $y:=v^3$ and $x:=uv$. Then $x^q+xy^{(q-1)/3}-y^{(2q+1)/3}=0$ and the extension $\mathbb{F}_{q^2}(u,v)/\mathbb{F}_{q^2}(x,y)$ is a Kummer extension of degree three in which the places below the pole and zero of $u$ are totally ramified. This implies immediately, by the Riemann-Hurwitz formula, that the function field $Y_3:=\mathbb{F}_{q^2}(x,y)$ has genus $g:=g(Y_3)=(q^2-q)/6$. Moreover, being a subfield of the Hermitian function field, the function field $Y_3$ is maximal over $\mathbb{F}_{q^2}$. 

We denote with $P_\infty$ and $P_{(0,0)}$ the two places of $Y_3$ lying below the zero and pole of $u$ in the Hermitian function field (we will refer to these places on the Hermitian function field as $Q_\infty$ and $Q_{(0,0)}$ respectively). Further, we denote by $P_{0}^i$, $i=1,\ldots,m$ the places of $Y_3$ lying below the remaining zeros of $v$. 
Since the divisor $\sum_{i=1}^{m} P_0^i$ will occur very often later on, we define
$$D_0:=\sum_{i=1}^{m} P_0^i.$$
Then one has the following divisors of $Y_3$:
\begin{equation} \label{divx1}
(x)=\frac{q+2}{3}P_{(0,0)}+ D_0-\frac{2q+1}{3}P_\infty,
\end{equation}
\begin{equation} \label{divy1}
(y)=P_{(0,0)} +3 D_0-q P_\infty,
\end{equation}

Let us denote by $\overline{\mathbb{F}}_{q^2}$ the algebraic closure of $\mathbb{F}_{q^2}$ and let $\Yte:=\overline{\mathbb{F}}_{q^2}Y_3$ be the function field obtained from $Y_3$ by extending the constant field to $\overline{\mathbb{F}}_{q^2}$. The places $P_{\infty},P_{(0,0)}$ and $P_0^i$ defined above give rise in a natural way to places of $\Yte$ and we will again denote them as $P_{\infty},P_{(0,0)}$ and $P_0^i$. All other places $P$ of $\Yte$ are uniquely determined by the values $a:=x(P) \in \overline{\mathbb{F}}_{q^2}$ and $b:=y(P)\in \overline{\mathbb{F}}_{q^2}$. We say that the place $P$ is centered in the point $(a,b)$ and will on occasion write $P_{(a,b)}$. This also explains the notation $P_{(0,0)}$ used previously. Similarly, the places of $\He:=\overline{\mathbb{F}}_{q^2}H_1$, except $Q_\infty$, can all be described as $Q_{(A,B)}$ (like we did for $Q_{(0,0)}$ previously).

For future computations, we will need a suitable canonical divisor and information about possible gaps of a place. We state these in the next lemma.

\begin{lemma}
\label{lemma:diffdy}
The divisor of the differential $dy$ on $Y_3$ is given by
\[(dy)=(m-1)(q+1)P_\infty+2D_0.\]
For any function $f \in L((m-1)(q+1)P_\infty+2D_0)$ and any place $P$ of $\Yte$ different from $P_\infty$ and the $P_0^i$, the value $v_P(f)+1$ is a gap of $P$, that is to say $v_P(f)+1 \in G(P).$
\end{lemma}

\begin{proof}
One the Hermitian function field $H_1$, we have $(dv)_{H_1}=(q^2-q-2)Q_\infty$. Hence 
\[(dy)_{H_1}=(dv^3)_{H_1}=2(v)_{H_1}+(q^2-q-2)Q_\infty=(q^2-3q-2)Q_\infty+2\sum_{A^q+A=0}Q_{(A,0)},\]
Where $Q_{(A,0)}$ denotes the zero of the function $v$ in the Hermitian function field satisfying $u(Q_{(A,0)})=A$. Since the different divisor of the function field extension ${H}_1/{Y_3}$ is equal to $2Q_{(0,0)}+2Q_\infty$, one finds that
$$\mathrm{Conorm}_{{H_1}/{Y_3}}((dy)_{Y_3})=(q^2-3q-4)Q_\infty+2\sum_{A^{q-1}+1=0}Q_{(A,0)}.$$
Hence
$$(dy)_{Y_3}=\frac{q^2-3q-4}{3}P_\infty+2\sum_{i=1}^{(q-1)/3}P_0^i=(m-1)(q+1)P_\infty+2\sum_{i=1}^{m}P_0^i=(m-1)(q+1)P_\infty+2D_0.$$
This shows that, for any $f \in L((m-1)(q+1)P_\infty+2D_0)$, the differential $fdy$ is holomorphic, i.e., has no poles. Then by \cite[Corollary 14.2.5]{VS} the integer $v_P(fdy)+1 \in G(P)$, for any place $P$. In particular if $P \neq P_\infty$ and $P \neq P_0^i$, then $v_P(f)+1 \in G(P)$.  
\end{proof}

Now let $P=P_{(a,b)}$ be a place of $\Yte$ different from $P_\infty$ and $P_0^i$. Further let $Q_{(A,B)}$ be a place of $\He$ lying above $P_{(a,b)}$. Note that $a=AB$ and $b=B^3$. We assume that $e(Q_{(A,B)}|P_{(a,b)})=1$, which only further excludes the place $P_{(0,0)}$. Note that $v-B$ is a local parameter of $Q_{(A,B)}$, which can be seen directly or by using that the extension $H_1/Y_3$ is only ramified at $P_{(0,0)}$ and $P_\infty$. Also the function $u-A$ is a local parameter of $Q_{(A,B)}$, since the equation $v^{q+1}=u^q+u$ implies that
\[(v-B)^{q+1}+B(v-B)^q+B^q(v-B)=(u-A)^q+(u-A).\] 
In particular $(u-A)=B^q(v-B) +O((v-B)^q)$, where $O((v-B)^q$ denotes terms of valuation at least $q$ at $Q_{(A,B)}$.

Since $e(Q_{(A,B)}|P_{(a,b)})=1$, we can use the local parameter $T:=(v-B)/B$ to compute power series developments of functions in $Y_3$. We start by computing such power series developments for $x_a:=(x-a)/a$ and $y_b:=(y-b)/b$. 
We have
\begin{eqnarray*}
x_a & = & (uv-AB)/(AB) \\
 & = & (B(u-A)+A(v-B)+(u-A)(v-B))/(AB)\\
 & = & ((B^{q+1}+A)(v-B)+B^q(v-B)^2)/(AB)+O(T^q)\\
 & = & (A^{q-1}+2)T+(A^{q-1}+1)T^2+O(T^q)\\
\end{eqnarray*}
and
\begin{eqnarray*}
y_b & = & (v^3-B^3)/B^3 \\
 & = & (3B^2(v-B)+3B(v-B)^2+(v-B)^3)/B^3\\
 & = & 3T+3T^2+T^3.
\end{eqnarray*}

The tangent line of the plane curve $x^q+xy^{(q-1)/3}=y^{(2q+1)/3}$ at the point $(a,b)$ is given by the zero set of the function
\[t_{P_{(a,b)}}:=3(x-a)-(a/b+b^{(q-1)/3})(y-b)=3ax_a-(a+b^{(q+2)/3})y_b.\]
Motivated by this, we compute the local power series development in $T$ of this function to determine its valuation at $P_{(a,b)}$:
\begin{eqnarray*}
t_{P_{(a,b)}} & = & 3AB((A^{q-1}+2)T+(A^{q-1}+1)T^2)-(AB+B^{q+2})(3T+3T^2+T^3)+O(T^q)\\
 & = & 3(A^qB+AB-B^{q+2})T+3(A^qB-B^{q+2})T^2-(AB+B^{q+2})T^3+O(T^q)\\
 & = & -3ABT^2-(2AB+A^qB)T^3+O(T^q).
\end{eqnarray*}
Since we assume $a \neq 0$, we can define 
\begin{equation}
    \label{eq:f0}
    f_0:=-t_{P_{(a,b)}}/a,
\end{equation}
giving
\begin{equation*}
    f_0=3T^2+(2+A^{q-1})T^3+O(T^q).
\end{equation*}
In particular, we may conclude that for $P_{(a,b)} \neq P_{(0,0)}$: 
\begin{equation*}
\label{eq:valf0}
v_{P_{(a,b)}}(t_{P_{a,b}})=v_{P_{(a,b)}}(f_0)=2.
\end{equation*}

It is possible to describe the power series of $x_a$, $y_b$ and $f_0$ more succinctly. Indeed, note that $A^{q-1}+1=B^{q+1}/A=B^{q+2}/(AB)=b^{(q+2)/3}/a=a^{q-1}/b^{(q-1)/3}+1.$ Motivated by this, we define 
\begin{equation}\label{def:alpha}
\alpha(P):= \left\{
\begin{array}{rl} 
b^{(q+2)/3}/a, & \quad P=P_{(a,b)}, a \neq 0,\\ 
0, & \quad P=P_0^i, \\ 
1, & \quad P=P_{(0,0)},\\ 
\infty, & \quad P=P_\infty.
\end{array}
\right. .
\end{equation} 
Then as long as $a\neq 0$, we obtain:
\begin{equation}\label{eq:xapowerseries}
x_a= (\alpha(P_{(a,b)})+1)T+\alpha(P_{(a,b)})T^2+O(T^q),
\end{equation}
\begin{equation}\label{eq:ybpowerseries}
y_b=3T+3T^2+T^3
\end{equation}
and
\begin{equation} 
    \label{eq:f0:exp}
f_0= 3T^2+(\alpha(P_{(a,b)})+1)T^3+O(T^q).
\end{equation}
Note that, since $x,y \in L(qP_\infty)$, we also have $x_a,y_b,f_0\in L(qP_\infty)$.

Finally, we remark the Fundamental Equation \cite[Page xix (ii)]{HKT}, which implies that there exists for any place $P$ of $\Yte$ a function $F_P$ in $\Yte$ such that 
\begin{equation} \label{fundeq1}
(F_P)=qP+\Phi(P)-(q+1)P_\infty,
\end{equation}
where $\Phi$ denotes the $q^2$-th power map, also known as the Frobenius map. Note that in case $P$ is $\mathbb{F}_{q^2}$-rational, one can find a function $F_P \in Y_3$ such that $(F_P)=(q+1)(P-P_\infty).$

Having established the existence of various functions, we proceed with the computation of Weierstrass semigroups at some special $\mathbb{F}_{q^2}$-rational places. 

\subsection{First observations on Weierstrass semigroups on \texorpdfstring{$Y_3$}{Y\_3}}

Our first aim is to compute the Weierstrass semigroup at the special places $P_{(0,0)}$, $P_\infty$ and $P_0^i$ with $i=1,\ldots,(q-1)/3$.

\begin{theorem} \label{sempinf}
$H(P_\infty)=H(P_{(0,0)})=\langle (2q+1)/3,q,q+1 \rangle$.
\end{theorem}

\begin{proof}
Using Equations \eqref{divx1} and \eqref{divy1} to compute the divisors of the functions $x$, $y$ and $x^3/y$, we already see that $(2q+1)/3,q,q+1 \in H(P_\infty)$. Let us denote by $H:=\langle (2q+1)/3,q,q+1 \rangle$ the numerical semigroup generated by $(2q+1)/3,q,q+1$. Since $H \subseteq H(P_\infty)$, the number of gaps of $H$, which we denote by $g$, is at least $g_3$. Therefore to show the theorem, it is enough to show that $g \le g_3$. However, this is a direct consequence of \cite[Korollar A.2.3]{Fuhrmann_1995}, where semigroups of the form $\langle a,q,q+1 \rangle$ were studied. Alternatively, see \cite[Lemma 3.4]{CKT}, where several results from \cite{Fuhrmann_1995} were paraphrased. 

Similarly, considering the functions $y/x^2$, $y^2/x^3$ and $y/x^3$, one obtains that $H(P_{(0,0)})=\langle (2q+1)/3,q,q+1 \rangle.$
\end{proof}

\begin{remark}
The fact that $H(P_\infty)=H(P_{(0,0)})=\langle (2q+1)/3,q,q+1 \rangle$ can also be deduced from part 5 of \cite[Proposition 5.6]{CKT2}, using that $P_\infty$ and $P_{(0,0)}$ are the places of ${Y}_3$ that are totally ramified in the function field extension $H_1/Y_3$ explained in the beginning of this section. 
\end{remark}

\begin{theorem}\label{thm:semgroup_P0i}
$H(P_{0}^i) = \langle q-2,q,q+1 \rangle$ for all $i=1,\ldots(q-1)/3$.
\end{theorem}

\begin{proof}
From Equation \eqref{fundeq1} the function $1/F_{P_{0}^i}$ gives that $q+1 \in H(P_{0}^i)$, while $y/F_{P_{0}^i}$ gives that $q-2 \in H(P_{0}^i)$. The fact that $q \in H(P_{0}^i)$ follows immediately by considering $x/F_{P_{0}^i}$.
Now we can proceed similarly as in the proof of Theorem \ref{sempinf} and cite \cite[Korollar A.2.3]{Fuhrmann_1995} or alternatively \cite[Lemma 3.4]{CKT}.
\end{proof}

\begin{remark} \label{rem:nonsymm1}
We note that $H(P_{0}^i)$ is not always symmetric. It is symmetric if $q=4$, as $3=(q^2-q)/3-1=2g_3-1$ cannot be written as a combination of multiples of $2$, $4$ and $5$. However, if $q \geq 7$ then $2g_3-1=1\cdot (q-2)+\frac{q-7}{3}\cdot q+1\cdot(q+1)$, implying that $2g_3 -1 \in H(P_{0}^i)$. Similarly, the semigroup $H(P_{(0,0)})=H(P_\infty)$ is not symmetric, since $2g_3-1=(2q+1)/3+(m-1)(q+1)$.
\end{remark}

The obtained results already show a difference between the places $P_{(0,0)}$ and $P_\infty$ on the one hand and the places $P_{0}^i$ on the other hand, provided that $q \neq 7$. When determining the automorphism group $\aut(Y_3)$, the key will be the observation that this group must act separately on $\mathcal{O}_{\infty}:=\{P_{(0,0)},P_\infty\}$, $\mathcal{O}_{0}:=\{P_{0}^i \mid 1 \le i \le (q-1)/3\}$ and the remaining $\mathbb{F}_{q^2}$-rational places of $Y_3$. This turns out to be a direct consequence of the fact that, for $P \in \mathcal{O}_{0} \cup \mathcal{O}_{\infty}$, the Weierstrass semigroup $H(P)$ is distinct from all the other Weierstrass semigroups of $\mathbb{F}_{q^2}$-rational places of $Y_3$. 

\begin{lemma}\label{lem:semgroupPrat}
Let $P \not\in \mathcal{O}_{0} \cup \mathcal{O}_{\infty}$ be an $\mathbb{F}_{q^2}$-rational place of $Y_3$. Then $q-1 \in H(P)$.
\end{lemma}

\begin{proof}
Since $P \not\in \mathcal{O}_{0} \cup \mathcal{O}_{\infty}$, we can write $P=P_{(a,b)}$ for $a \neq 0$. Now consider the function $f_0$ defined in Equation \eqref{eq:f0}. Since $f_0 \in L(q P_\infty)$ and $v_P(f_0)=2$, the function 
$$t_{q-1}:=\frac{t_{P}}{F_P}$$
has divisor of the form
$$(t_{q-1})=-(q-1)P+E+P_\infty,$$
where $E$ is an effective divisor not containing $P$.
\end{proof}

\section{The full automorphism group of \texorpdfstring{$Y_3$}{Y\_3}}

In this subsection, we will determine the full automorphism group $G$ of the function field $Y_3$. We will start by computing those automorphism that can be induced by automorphisms of the Hermitian function field $H_1:=\mathbb{F}_{q^2}(u,v)$, with $v^{q+1}=u^q+u$. As explained in the beginning of this section, the function field $Y_3=\mathbb{F}_{q^2}(x,y)$ can be seen as a subfield of $\mathbb{F}_{q^2}(u,v)$ of index three simply by choosing $y:=v^3$ and $x:=uv$.

The extension $\mathbb{F}_{q^2}(u,v)/\mathbb{F}_{q^2}(x,y)$ is a Kummer extension of degree, where the places $P_{(0,0)}$ and $P_\infty$ of $Y_3$ are totally ramified and we denote by $Q_{(0,0)}$ and $Q_\infty$ places of $H$ lying above them. The Galois group of this function field extension is generated by the automorphism $\tau: (u,v) \mapsto  (\zeta_3 u,\zeta_3^2 v)$ where $\zeta_3$ is a primitive third root of unity. 

The normalizer $N(\tau)$ of $\langle \tau \rangle$, the order three subgroup generated by $\tau$ in $\PGU(3,q)$, is well understood. This group acts transitively on the set $\{Q_{(0,0)},Q_\infty\}$. Indeed, the automorphism
$$i: (u,v) \mapsto \bigg( \frac{1}{u}, \frac{v}{u}\bigg)$$
of $H_1$ has order two and interchanges $Q_{(0,0)}$ and $Q_\infty$. It also normalizes $\langle \tau \rangle$, since 
$i \circ \tau \circ i= \tau^2$ as can be checked by a direct computation.

From for example \cite{GSX}, we see that the intersection of the stabilizers of $Q_{(0,0)}$ and $Q_\infty$ is a cyclic group $S$ of order $q^2-1$ generated by $\sigma: (u,v) \mapsto (\mu^{q+1} u,\mu v)$ where $\langle \mu \rangle = \mathbb{F}_{q^2}^*$. Since $\tau \in \langle \sigma \rangle$, clearly $\sigma \in N(\tau)$. Hence, from the orbit-stabilizer theorem, $|N(\tau)|=2(q^2-1)$. Note that the fixed field of $S$ is the rational function field $\mathbb{F}_{q^2}(u^{q-1})=\mathbb{F}_{q^2}(v^{q+1}/u)=\mathbb{F}_{q^2}(y^{(q+2)/3}/x).$

We have shown that $N(\tau)=\langle S,i \rangle$ and that it turns out that it is isomorphic to the dihedral group of order $2(q^2-1)$. We can therefore predict that the quotient group $N(\tau)/\langle \tau \rangle$ is an automorphism group of $\mathbb{F}_{q^2}(x,y)$ of order $2m$. This is how the automorphism group $G$ described in the following proposition has been computed.

\begin{proposition} \label{subgroup}
Let $M:=(q^2-1)/3$. The function field $Y_3$ has a dihedral automorphism group $G:=C_{M} \rtimes C_2 \cong D_{M}$ of order $2M$. The group $C_M$ is generated by the automorphism 
$$\bar\sigma: (x,y) \mapsto (\delta^{\frac{q+2}{3}}x,\delta y),$$
where $\delta$ is a primitive $M$-th root of unity, while $C_2$ is generated by the automorphism $\bar i$ given by
$$\bar i: (x,y) \mapsto \bigg( \frac{y}{x^2}, \frac{y^2}{x^3}\bigg).$$
\end{proposition}

\begin{proof}
The fact that $\bar\sigma$ and $\bar i$ are automorphisms of $Y_3$ of orders $M$ and $2$ can be checked directly by hand.
One can also check by hand that $\bar i \bar\sigma \bar i \in S$, which shows that $C_2=\langle \bar i \rangle$ normalizes $C_M=\langle \bar \sigma \rangle$.
The fact that $\langle C_M, C_2 \rangle=C_M \rtimes C_2$ now follows, since $C_M \cap C_2=\{id\}$.
\end{proof}

As we will see in the following, the cases $q=4$ and $q=7$ are a bit special, and as such they will be analyzed separately in the following observation. This will allow us for the rest of the section to assume that $q \geq 13$ without loss of generality.

\begin{observation} \label{q7}
Let $q=7$. Then the automorphism group of $Y_3$ can be computed using for example the computer algebra package MAGMA. Its order is $64$, that is, two times larger than $|G|$. If $q=4$ MAGMA can be used again to compute the automorphism group of $Y_3$. In this last case, its order is $160$. This example was already pointed out by Henn \cite[Theorem 11.127 (I)]{HKT} as one of the few examples of a function field of genus $g$ with at least $8g^3$ automorphisms.
\end{observation}

From now on we will so always assume $q \geq 13$. Our aim is to show that in this case $G=\aut(Y_3)$. To do so, we show how the information deduced in the previous subsection allows us to predict some orbits of $\aut(Y_3)$ on its action on the places of $Y_3$.

 \begin{lemma} \label{orbits}
The sets $\mathcal{O}_\infty=\{P_{(0,0)},P_\infty\}$ and $\mathcal{O}_{0}=\{P_0^i \mid 1 \le i \le (q-1)/3\}$ are orbits of $\aut({Y}_3)$. Furthermore, $\aut({Y}_3)$ acts on the set of $\mathbb{F}_{q^2}$-rational places of $Y_3$ not in $\mathcal{O}_\infty \cup \mathcal{O}_0$.
\end{lemma}

\begin{proof}
Since ${Y}_3$ is $\mathbb{F}_{q^2}$-maximal, its full automorphism group $\aut({Y}_3)$ is defined over $\mathbb{F}_{q^2}$ and hence acts on the set of $\mathbb{F}_{q^2}$-rational places of $Y_3$, see for example \cite[Lemma 2.4]{BMT}.
The fact that $\aut({Y}_3)$ must act separately on $\mathcal{O}_\infty$ and $\mathcal{O}_0$ follows from Theorem \ref{sempinf}, Theorem \ref{thm:semgroup_P0i} and Lemma \ref{lem:semgroupPrat}. Note that this requires $q \geq 13$. Indeed $q-2=(2q+1)/3$ if $q=7$, while $q-1=(2q+1)/3$ for $q=4$.

Since $\bar i(P_{(0,0)})=P_\infty$, $\mathcal{O}_{\infty}$ is an orbit of $\aut({Y}_3)$ for $q \geq 13$. Note that the fixed field of $C_M =\langle \bar\sigma \rangle$ is equal to the fixed field of $S =\langle \sigma \rangle$, which is $\mathbb{F}_{q^2}(y^{(q+2)/3}/x)$. The places $P_0^i$ are exactly the $(q-1)/3$ places of $Y_3$ lying above the zero of $y^{(q+2)/3}/x$ in the Kummer extension $\mathbb{F}_{q^2}(x,y)/\mathbb{F}_{q^2}(y^{(q+2)/3}/x)$. Therefore $C_M$ acts transitively on these places, implying that $\mathcal{O}_0$ is an orbit of $\aut(Y_3)$.

Again using that the automorphism group $\aut({Y}_3)$ is defined over $\mathbb{F}_{q^2}$, we can now conclude that it acts on the set of $\mathbb{F}_{q^2}$-rational places of $Y_3$ not in $\mathcal{O}_\infty \cup \mathcal{O}_0$.
\end{proof}

\begin{theorem}
Assume that $q \geq 13$ and $q \equiv 1 \pmod 3$. Then $\aut(Y_3)=G$. In particular, $\aut(Y_3)$ is a dihedral group of order $2M$.
\end{theorem}

\begin{proof}
Since $\mathcal{O}_{\infty}=\{P_{(0,0)},P_\infty\}$ is an orbit of  $\aut(Y_3)$, we have from the orbit stabilizer theorem that
$$|\aut(Y_3)|=|\mathcal{O}_\infty| |\aut(Y_3)_{P_\infty}|=2|\aut(Y_3)_{P_\infty}|,$$
 where $\aut(Y_3)_{P_\infty}$ denotes the stabilizer of $P_\infty.$
Since $C_M=\langle \bar\sigma \rangle \subseteq \aut(Y_3)_{P_\infty}$, we also see that $M$ divides $|\aut(Y_3)_{P_\infty}|$. 

We first note that the characteristic $p$ of $\mathbb{F}_{q^2}$ cannot divide $|\aut({Y}_3)_{P_\infty}|$. Indeed, if by contradiction there existed an element $\theta \in \aut({Y}_3)_{P_\infty}$ of order $p$, then $\theta$ would have exactly one fixed point, because it is a $p$-element and $Y_3$ has $p$-rank zero, being a maximal function field. However, since $\theta$ acts on $\mathcal{O}_{\infty}=\{P_{(0,0)},P_\infty\}$ fixing $P_\infty$, $\theta$ must fix $P_{(0,0)}$ as well, which is a contradiction.
 This proves that $\aut(Y_3)_{P_\infty}$ has order coprime to $p$, say $|\aut(Y_3)_{P_\infty}|=M\ell$, with $\ell \geq 1$. 
 From \cite[Theorem 11.60]{HKT} 
  $$|\aut(Y_3)_{P_\infty}|=M\ell=\frac{(q^2-1)\ell}{3} \leq 4g+2=\frac{2q^2-2q+6}{3}.$$
Hence, since $q \geq 13$:
$$\ell \leq \bigg\lfloor\frac{2q^2-2q+6}{q^2-1} \bigg\rfloor=1.$$
This shows that $|\aut(Y_3)|=|\mathcal{O}_0| |\aut(Y_3)_{P_\infty}|=2M$, and thus $\aut(Y_3)=G$.
\end{proof}

\section{The \texorpdfstring{$\mathcal{P}$}{P}-order of a place \texorpdfstring{$P$}{P}}

In this section, we prepare for the computation of the semigroups at all places $P$ of $\Yte$ not in $\mathcal{O}_0 \cup \mathcal{O}_\infty$. In \cite{BMV} a similar task for another maximal curve was carried out, but as we will see the situation is significantly more complicated in the present situation. Nonetheless, some techniques from \cite{BMV} can be used. For this reason, we only recall the definitions and statements here, referring to \cite{BMV} for the proofs and further details.

\begin{definition}[\cite{BMV}]\label{def:PQ}
Let $i \in \Z$. Further, let $\F$ be a field of characteristic different from three and assume that it contains a primitive cube root of unity, which we will denote by $\zeta_3$. Then we define the following rational functions in $\F(s)$:
\begin{equation*}
    \PP_i(s):=\frac{(s + \zeta_3)^{3i} - (s + \zeta_{3}^2)^{3i}}{3(\zeta_3 - \zeta_{3}^2)s(s-1)}
\end{equation*}
and
\begin{equation*}
    \QQ_{i}(s):=\frac{\left(\frac{1-\zeta_3}{3}\right)(s + \zeta_3)^{3i-1} + \left(\frac{1-\zeta_{3}^2}{3}\right)(s + \zeta_{3}^2)^{3i-1}}{s-1}.
\end{equation*}
\end{definition}

\begin{lemma}[\cite{BMV}]
\label{lem:PQ:degree}
Let $i \in \Z_{>0}$. Then $\PP_i(s)$ is a nonzero polynomial of degree at most $3i-3$, while $\QQ_{i}(s)$ is a nonzero polynomial of degree $3i-2$.
\end{lemma}

\begin{lemma}[\cite{BMV}]
\label{lem:PQ:identities}
Let $i,j,\ell \in \Z$. Then
\begin{equation*}
    \PP_i(s)\PP_{\ell + j}(s)-\PP_j(s) \PP_{\ell+i}(s) =(s^2-s+1)^{3j} \PP_{i-j}(s) \PP_{\ell}(s)
\end{equation*}
and
\begin{equation*}
    \PP_i(s) \QQ_{\ell+j}(s) -\PP_j(s) \QQ_{\ell+i}(s) =(s^2-s+1)^{3j}\PP_{i-j}(s) \QQ_{\ell}(s).
\end{equation*}
\end{lemma}

\begin{remark}[\cite{BMV}]
\label{rem:summary}
The following facts hold:
\begin{enumerate}
    \item The coefficient of $s^{3i-3}$ of the polynomial $\PP_i(s)$ equals $3i(\zeta_3-\zeta_3^2)$. Hence if the characteristic of the field $\F$, which already is assumed to be distinct from three, is zero or does not divide $i$, then the degree of $\PP_i(s)$ is exactly $3i-3$. Since we will work over the finite field $\F_{q^2}$, where $q \equiv 1 \pmod{3}$, it may well happen that $\deg \PP_i(s) < 3i-3.$
    \item For any $i \in \Z_{>0}$, the polynomials $\PP_i(s)$ and $\QQ_i(s)$ have no common roots.
    \item Let $\F=\overline{\F}_{q^2}$ be the algebraic closure of $\F_{q^2}.$ Then for any $\alpha \in \F \setminus \{0,1,-\zeta_3,-\zeta_3^2\}$, there exists $i>0$ such that $\PP_{i+1}(\alpha)=0$.
\end{enumerate}
\end{remark}

In light of the third item of Remark \ref{rem:summary}, the following definition is well posed:

\begin{definition}[\cite{BMV}]
\label{def:Pord}
Let $\alpha \in \overline{\F}_{q^2} \setminus \{0,1,-\zeta_3,-\zeta_3^2\}$. Then we define the $\PP$-order of $\alpha$ as the smallest positive integer $i$ such that $\PP_{i+1}(\alpha)=0$. If $P$ is a place of $\Yte$ such that $\alpha(P) \in \overline{\F}_{q^2} \setminus \{0,1,-\zeta_3,-\zeta_3^2\}$, then define the $\PP$-order of $P$ to be the $\PP$-order of $\alpha(P)$.
\end{definition}

We now start investigating the notion of $\PP$-order in connection with places of the function fields $Y_3$ and $\Yte$. 

\begin{lemma}\label{lem:q+1stpower}
Let $P$ be an $\mathbb{F}_{q^2}$-rational place of $Y_3$ not in $\mathcal{O}_0 \cup \mathcal{O}_\infty$ and suppose that $\alpha(P)^2-\alpha(P)+1 \neq 0$. Then \[ \alpha(P)^q-\alpha(P)^{q-1}=1 \quad \text{and} \quad \left(\frac{\alpha(P)+\zeta_3}{\alpha(P)+\zeta_3^2} \right)^{q+1} = \zeta_3.\]
In particular, if the $\PP$-order of $\alpha(P)$ equals $i$,  then $i+1$ divides $q+1.$ 
\end{lemma}

\begin{proof}
Since $P \not\in \mathcal{O}_0 \cup \mathcal{O}_\infty$, we can write $P=P_{(a,b)}$ with $a \neq 0$. In particular Equation \eqref{def:alpha} implies that $\alpha(P) \not\in \{0,1,\infty\}$. Moreover, since $P$ is a rational place, we know $a,b \in \mathbb{F}_{q^2}$. For convenience, we write $\alpha$ instead of $\alpha(P)$ in the remainder of this proof. Note that from the definition of $\alpha$ and using $b^{q^2}=b$, we see that
\begin{eqnarray*}
\alpha^q & = & \left(\frac{b^{(q+2)/3}}{a}\right)^q=\frac{b^{(q^2+2q)/3}}{a^q} =  \frac{b^{(2q+1)/3}}{a^q} = \frac{a^q+a b^{(q-1)/3}}{a^q}= \frac{a^{q-1}/b^{(q-1)/3}+1}{a^{q-1}/b^{(q-1)/3}} = \frac{\alpha}{\alpha-1}.
\end{eqnarray*}
This implies the first equation.
Further, since $\alpha^2-\alpha+1 \neq 0$, we obtain 
\begin{eqnarray*}
\left(\frac{\alpha+\zeta_3}{\alpha+\zeta_3^2}\right)^q & = & \frac{\alpha^q+\zeta_3^q}{\alpha^q+\zeta_3^{2q}}=\frac{\alpha/(\alpha-1)+\zeta_3}{\alpha/(\alpha-1)+\zeta_3^{2}}=\frac{(1+\zeta_3)\alpha-\zeta_3}{(1+\zeta_3^2)\alpha-\zeta_3^2}=\zeta_3\frac{\alpha+\zeta_3^2}{\alpha+\zeta_3}.
\end{eqnarray*}
The lemma now follows.
\end{proof}

\begin{lemma}\label{lem:case1rationalpoints}
Let $P$ be a place of $\Yte$ not in ${\mathcal O}_{\infty} \cup {\mathcal O}_{0}$ and suppose that $\alpha(P)^2-\alpha(P)+1 = 0$. Then $P$ can be identified with an $\mathbb{F}_{q^2}$-rational place of $Y_3$ and there are exactly $2(q^2-1)/3$ possibilities for such $P$. 

If $P$ instead satisfies $\alpha(P)^q-\alpha(P)^{q-1}= 1$, then $P$ can also be identified with an $\mathbb{F}_{q^2}$-rational place of $Y_3$ and for each such $\alpha(P)$ there exist exactly $(q^2-1)/3$ possibilities for $P$. 
\end{lemma}

\begin{proof}
If $P \not\in {\mathcal O}_{\infty} \cup {\mathcal O}_{0}$, then $\alpha(P) \not\in \{0,1,\infty\}.$ Further, if $P$ comes from an $\mathbb{F}_{q^2}$-rational place of $Y_3$, then Lemma \ref{lem:q+1stpower} implies that we have two possibilities: $\alpha(P)^2-\alpha(P)+1=0$ or $\alpha(P)^q-\alpha(P)^{q-1}=1$. Hence we have a total of $q+2$ possibilities for $\alpha(P)$. First we show the claim that for each of these possibilities, there are at most $(q^2-1)/3$ possibilities for $P$. Indeed, since $A^{q-1}=\alpha(P)-1$ and $B^{q+1}=A^q+A$, there are at most $q^2-1$ possibilities for places $Q_{(A,B)}$ of the Hermitian function field $\He$ lying above the possible $P$'s. Since $\alpha(P) \not\in \{0,1,\infty\}$, we have $e(Q_{(A,B)}|P)=1$ is unramified. This means that there are $(q^2-1)/3$ possibilities for $P$. In particular, there are at most $(q^2-1)/3$ possibilities for $P$ just as claimed.

Now denote by $N(Y_3)$, the number of $\mathbb{F}_{q^2}$-rational places of $Y_3$. By the above, we see that
\[N(Y_3) \le |{\mathcal O}_{\infty}|+ |{\mathcal O}_{0}|+(q+2)\frac{q^2-1}{3}=2+\frac{q-1}{3}+(q+2)\frac{q^2-1}{3}=\frac{q^3+2q^2+3}{3}.\]
However, since $Y_3$ is a maximal function field over $\mathbb{F}_{q^2}$ of genus $(q^2-q)/6$, we also have 
\[N(Y_3)=q^2+1+2 q \frac{q^2-q}{6}=\frac{q^3+2q^2+3}{3}.\]
Hence in our upper bound for $N(Y_3)$ equality needs to hold. This implies that any place $P$ of $\Yte$ such that $\alpha(P)^2-\alpha(P)+1=0$ can be identified with a rational place of $Y_3$. The lemma now follows.
\end{proof}

The following lemma strengthens the previous one and involves $\PP$-orders. We denote by $\varphi(\cdot)$ the Euler totient function.

\begin{lemma}\label{lem:case2rationalpoints}
Let $P$ be a place of $\Yte$ not in ${\mathcal O}_{\infty} \cup {\mathcal O}_{0}$ and suppose that $\alpha(P)^2-\alpha(P)+1 \neq 0$ and that $i$, the $\PP$-order of $\alpha(P)$, is such that $i+1$ divides $q+1$. Then there are exactly $(q^2-1)\varphi(i+1)$ many possibilities for $P$. Moreover, exactly
$\frac{q^2-1}{3}\varphi(i+1)$ many of these can be identified with $\mathbb{F}_{q^2}$-rational places of $Y_3$.
\end{lemma}

\begin{proof}
Let $q+1=(i+1)\cdot d$. There are $\varphi(i+1)$ possible values of $\beta:=((\alpha+\zeta_3)/(\alpha+\zeta_3^2))^3$ that have multiplicative order $i+1$. For each such $\beta$, one has
\[\left(\frac{\alpha+\zeta_3}{\alpha+\zeta_3^2}\right)^{i+1}=1 \quad \text{or} \quad  \left(\frac{\alpha+\zeta_3}{\alpha+\zeta_3^2}\right)^{i+1}=\zeta_3 \quad  \text{or} \quad  \left(\frac{\alpha+\zeta_3}{\alpha+\zeta_3^2}\right)^{i+1}=\zeta_3^2,\]
implying that
\[\left(\frac{\alpha+\zeta_3}{\alpha+\zeta_3^2}\right)^{q+1}=1 \quad \text{or} \quad \left(\frac{\alpha+\zeta_3}{\alpha+\zeta_3^2}\right)^{q+1}=\zeta_3^d \quad \text{or} \quad \left(\frac{\alpha+\zeta_3}{\alpha+\zeta_3^2}\right)^{q+1}=\zeta_3^{2d}.\]
Note that $3$ does not divide $d$, since $q+1 \equiv 2 \pmod{3}.$ 
According to Lemma \ref{lem:q+1stpower}, exactly one of these can give rise to values of $\alpha$ equal to $\alpha(P)$ for some $\mathbb{F}_{q^2}$-rational place $P$. The remaining two give rise to $2(q^2-1)/3$ non-rational places. 
Reasoning as in the proof of Lemma \ref{lem:case1rationalpoints}, we see that there are at most $\varphi(i+1)(q^2-1)/3$ many $\mathbb{F}_{q^2}$-rational places $P$ giving rise to $\alpha(P)$ with $\mathcal P$-order $i$.
Since $\sum_{i>0;i+1|q+1} \varphi(i+1)=q+1-1=q,$ we see that there are at most $q(q^2-1)/3$ many $\mathbb{F}_{q^2}$-rational places $P$ of $Y_3$  satisfying $\alpha(P)^2-\alpha(P)+1 \neq 0$ and $P \not \in {\mathcal O}_{\infty} \cup {\mathcal O}_{0}.$ Again reasoning as in the proof of Lemma \ref{lem:case1rationalpoints}, we see that there are exactly $q(q^2-1)/3$ many such $\mathbb{F}_{q^2}$-rational places. We conclude that, for each $i \ge 1$ such that $i+1$ divides $q+1$, there are exactly $\varphi(i+1)(q^2-1)/3$ $\mathbb{F}_{q^2}$-rational places of $Y_3$ such that the $\mathcal P$-order of $\alpha(P)$ is $i$.
\end{proof}

Now that we have characterized the $\PP$-orders of rational places as well as how many such places have a specific $\PP$-order, we continue with the main investigation of this paper: the semigroups at places of $\Yte$. We begin by computing the semigroup at any rational place of $Y_3$.

\section{Weierstrass semigroups at the \texorpdfstring{$\mathbb{F}_{q^2}$}{F\_q2}-rational places of \texorpdfstring{$Y_3$}{Y\_3}}
\label{sec:rational}

The aim of this section is to compute the Weierstrass semigroup at any rational place of $Y_3$. In order to do this, we will construct two families of functions, in Theorems \ref{thm:fi} and \ref{thm:gi}, that will be crucial for this computation. In the remainder of this article, we write $m:=(q-1)/3$.

\begin{remark}
Theorems \ref{thm:fi} and \ref{thm:gi} presented below are the counterpart for the function field $Y_3$ of Theorems 3.12 and 3.19 from \cite{BMV}. The proofs are in fact very similar. The main differences between the proofs in \cite{BMV} and the proofs given below is that the roles of $x_a$ and $y_b$ are interchanged and that $m$ now stands for $(q-1)/3$. 
\end{remark}

\begin{theorem}
\label{thm:fi}
Let $P$ be a place of $\Yte$ not in $\mathcal{O}_\infty \cup \mathcal{O}_0$ and suppose that $\al(P)^2-\al(P) +1 \neq 0$. Further, let $i$ be the $\PP$-order of $\al(P)$. If $i \le m-1$, then there exists a function $f_{i}\in L((i+1)qP_{\infty})$ such that $v_{P}(f_{i})=3i+3$. Moreover, for each $j\in \mathbb{Z}$ with $0\leq j\leq \min\{i-1,m-1\}$, there exists a function $f_j\in L((j+1)qP_{\infty})$ with $v_{P}(f_j)=3j+2$.
\end{theorem}

\begin{proof}
Throughout the proof we simplify the notation by writing $\al$ instead of $\al(P)$. In a similar vein, we will write $\mathcal{P}_{j}$ and $\QQ_{j}$, rather than $\mathcal{P}_{j}(\al)$ and $\QQ_{j}(\al)$.

Writing $P=P_{(a,b)}$, let $Q=Q_{(A,B)}$ be a place of $\He$ lying over $P_{(a,b)}$ and let $T:=\frac{v-B}{B}$, which is a local parameter at $Q_{(A,B)}$.
For each $j$ such that $0\leq j\leq i$, we claim that there exists a function $f_j\in L((j+1)qP_{\infty})$ such that the local power series expansion of $f_j$ at $Q_{(A,B)}$ with respect to the local parameter $T$ is
\begin{equation}
\label{eq:fj}
f_j = 3\mathcal{P}_{j+1} T^{3j+2} + \QQ_{j+1}  T^{3j+3} + O(T^{q}).
\end{equation}
Note that by definition of the $\PP$-order, this will imply that
\begin{equation}
\label{eq:fi}
f_i = \QQ_{i+1}  T^{3i+3} + O(T^{q}).
\end{equation}
This is sufficient to prove the theorem since $v_{Q_{(A,B)}}(f_j)=v_{P_{(a,b)}}(f_j)$ and $3j+3 <q$ for all $j$ under consideration.

First of all, note that, for $j=0$, we can take $f_0$ to be exactly the function defined in Equation \eqref{eq:f0} and whose local power series expansion with respect to $T$ was computed in Equation \eqref{eq:f0:exp}.
To show the result for $j=1$, we define
\begin{equation*}
f_1 := -9y_b^2 +27f_0 -3(\al-5)y_b f_0+ (\al^2 -\al-5)f_0^2.
\end{equation*}
Elementary calculations show that the local power series expansion of $f_1$ at $Q_{(A,B)}$ with respect to $T$ is precisely
\begin{equation*}
    f_1 = 3\PP_2 T^5 + \QQ_2 T^6 + O(T^{q}).
\end{equation*}
For $j=2$, we define
\begin{equation*}
    f_2:=(\al+1)^{-3}\left(-27\PP_2f_1+3\PP_2^2f_0^2y_b-3\PP_2(\al^4+\al^3-4\al^2-4\al+3)f_0^3\right)+(7\al^2-16\al+7)f_1f_0.
\end{equation*}
A lengthy but not difficult calculation now shows that the local power series expansion of $f_2$ equals
\begin{equation*}
    f_2 = 3\PP_3 T^{8} + \QQ_3 T^{9} + O(T^{q}).
\end{equation*}

For $3\leq j \leq i$, we assume now that $f_{j-1}$ and $f_{j-2}$ have the form claimed in Equation \eqref{eq:fj} and we construct inductively the remaining functions $f_j$ in the following way, defining:
\begin{equation*}
    f_j:=-\frac{\PP_j f_{j-2}f_1 - \PP_2\PP_{j-1} f_{j-1}f_0}{(\al^2 -\al +1)^2 \PP_{j-2}}.
\end{equation*}
The idea of choosing the functions $f_{j-2}f_1$ and $f_{j-1}f_0$ is that the vanishing order at $Q_{(A,B)}$ is $3j+1$ for both. Hence, a suitable linear combination of them will vanish with order at least $3j+2$. Moreover, as $f_{j-2}f_1$ and $f_{j-1}f_0$ lie in $L((j+1)qP_\infty)$, a linear combination of them does as well. Therefore, we only need to show that
\begin{equation*}
    \PP_j f_{j-2}f_1 - \PP_2\PP_{j-1} f_{j-1}f_0 = -(\al^2 -\al +1)^2 \PP_{j-2}\left(3\mathcal{P}_{j+1} T^{3j+2} + \QQ_{j+1}  T^{3j+3} + O(T^{q})\right).
\end{equation*}
To show this equality, the proof of Theorem 3.12 from \cite{BMV} can be copied verbatim. 
Equation \eqref{eq:fj} now follows directly, while Equation \eqref{eq:fi} follows by observing that $\PP_{i+1} = 0$ by hypothesis and $\QQ_{i+1} \neq 0$ by the second part of Remark \ref{rem:summary}.
As we have already observed that, by construction, $f_j\in L((j+1)qP_{\infty})$ for all $j$ in $0\leq j\leq i$, the proof of the theorem is complete.
\end{proof}

Since we divided by $\alpha(P)^2-\alpha(P)+1$ in the proof of the theorem, we need to redo the construction of special functions in case $\alpha(P)^2-\alpha(P)+1=0$. We do this in the following theorem.

\begin{theorem}
\label{thm:gi}
Let $P$ be a place of $\Yte$ such that $\al(P)^2-\al(P)+1=0$. Then, for every non-negative integer $i$ such that $i \le m-1$, there exists a function $\widetilde{f}_i\in L((i+1)qP_{\infty})$ with $v_{P_{(a,b)}}(\widetilde{f}_i)=3i+2$.
\end{theorem}

\begin{proof}
As before, in this proof we write $P=P_{(a,b)}$, $\al$ instead of $\al(P_{(a,b)})$ and $\mathcal{P}_{j}$, $\QQ_{j}$ instead of $\mathcal{P}_{j}(\al)$, $\QQ_{j}(\al)$. For each $i\in \mathbb{Z}_{\geq 0}$, we claim that there exists a function $\widetilde{f}_i\in L((i+1)qP_{\infty})$ such that the local power series expansion of $\widetilde{f}_i$ at $Q_{(A,B)}$ with respect to the local parameter $T$ is:
\begin{equation}\label{eq:gi}
\widetilde{f}_i = 3T^{3i+2} + (\alpha+1)T^{3i+3} + O(T^{q})
\end{equation}
Denoting by $f_0$ and $f_1$ the functions constructed in the previous theorem, we see that we can choose $\widetilde{f}_0=f_0$ and $\widetilde{f}_1=(2\alpha-1)f_1/9$, since $$(2\alpha-1)3\PP_2 \equiv 27 \pmod{\alpha^2-\alpha+1}$$ and $$(2\alpha-1)\QQ_2 \equiv 9\alpha+9 \pmod{\alpha^2-\alpha+1}.$$

For $i \geq 2$, we assume now that $\widetilde{f}_{i-1}$ and $\widetilde{f}_{i-2}$ have the form claimed in Equation \eqref{eq:gi} and we construct inductively the remaining functions $\widetilde{f}_i$ by taking a suitable linear combination of
\begin{equation*}
    \widetilde{f}_{i-1}, \quad \widetilde{f}_{i-2}\cdot \widetilde{f}_0\cdot y_b, \quad \widetilde{f}_{i-2}\cdot \widetilde{f}_0^2 \quad \mbox{and} \quad \widetilde{f}_{i-1}\cdot \widetilde{f}_0.
\end{equation*}
In fact, the vanishing orders of these four functions at $Q_{(A,B)}$ are $3i-1$, $3i-1$, $3i$ and $3i+1$ respectively, therefore a suitable linear combination of them will vanish with order at least $3i+2$. Moreover, since the four functions all lie in $L((i+1)qP_{\infty})$, any linear combination of them does as well.

More in detail, if we set
\[\widetilde{f}_i:=(6\alpha-3)\widetilde{f}_{i-1}-\frac{2\alpha-1}{3}\widetilde{f}_{i-2}\widetilde{f}_0y_b+\frac{3\alpha-2}{3}\widetilde{f}_{i-2}\widetilde{f}_0^2-(\alpha-2) \widetilde{f}_{i-1}\widetilde{f}_0,\]
then a direct computation shows that Equation \eqref{eq:gi} is satisfied.
\end{proof}

The constructed functions $f_j$ and $\widetilde{f}_j$ are enough to be able to compute the Weierstrass semigroups of all $\mathbb{F}_{q^2}$-rational places $P$ of $Y_3$. Note that, for rational places $P$, Equation \eqref{fundeq1} implies that there exists a function $F_P \in Y_3$ such that $(F_P)=(q+1)(P-P_\infty)$. It is well known that the semigroup of a rational place of a maximal function field over $\mathbb{F}_{q^2}$ contains $q$ and $q+1$. For the rational places $P=P_{(a,b)}$ with $a,b \neq 0$, which we are now assuming, this is actually easy to see. Simply consider the function $y_b/F_P$, using Equation \eqref{eq:ybpowerseries}, and the function $1/F_P$. We now consider the case where $\alpha(P)^2-\alpha(P)+1=0$.

\begin{theorem}
\label{thm:special:short:orbit}
Let $P$ be a place of $Y_3$ such $\alpha(P)^2-\alpha(P)+1=0$. Then
\begin{equation*}
    H(P) = \langle q, q+1, (q-1) +j(q-2) \mid j=0,\ldots, m-1\rangle.
\end{equation*}
\end{theorem}
\begin{proof}
We already know that $P$ is a rational place and that $P \not \in \mathcal{O}_0\cup \mathcal{O}_\infty$. In particular, by Equation \eqref{fundeq1}, there exists a function $F_P \in Y_3$ such that $(F_P)=(q+1)(P-P_\infty)$.

We start by showing that the semigroup $H:=\langle q, q+1, (q-1) +j(q-2) \mid j=0,\ldots, m-1\rangle$ is contained in $H(P)$ and subsequently we prove that it has at most $g(Y_3)$ gaps.
We already know that $q,q+1 \in H(P)$. Let $F_{P}$ be as in Equation \eqref{fundeq1} and, for all $j$ such that $j=0,\ldots, m-1$, define the function
\begin{equation*}
    F_j:= \frac{\widetilde{f}_j}{F_P^{j+1}},
\end{equation*}
where the $\widetilde{f}_j$ are the functions constructed in Theorem \ref{thm:gi}. Note that since $P$ is a rational place, we have $(F_P)=(q+1)(P-P_\infty)$. Therefore $(F_j)=E_j-((q-1)+j(q-2))P$, where $E_j$ is an effective divisor whose support does not contain $P$. This shows that for $j=0,\dots,m-1$, the semigroup $H(P)$ contains $(q-1)+j(q-2)$. 

Now all we need to do is to check that the semigroup $H$ has $g(Y_3)$ many gaps at most.
We know $0 \in H$, but we claim that for $j=1,\dots,m-1$, all integers in $\{j(q-2)+1,\dots,j(q+1)\}$ are in $H$ as well. This is clear for $j=1$, since $q-1,q,q+1 \in H$. If this is true for some $j<m$, then adding $q-1$ and $q+1$ to all integers in $\{j(q-2)+1,\dots,j(q+1)\}$ shows that the consecutive integers in $\{(j+1)(q-2)+2,\dots,(j+1)(q+1)\}$ are all in $H$. Since $(j+1)(q-2)+1=(q-1)+j(q-2) \in H$, we conclude that all integers in $\{(j+1)(q-2)+1,\dots,(j+1)(q+1)\}$ are in $H$. This shows the claim. Now note that $\{m(q-2)+1,\dots,m(q+1)\}$ consists of $q-1$ consecutive integers, all in $H$. Adding integral multiples of $q-1$ and $q$ to this set, we obtain that all integers greater than or equal to $m(q-2)+1+q-1=(m+1)(q-2)+2 = m(q+1)+1$ are in $H$. This means that the number of gaps in $H$ is at most \[g(H) \le (q-2)+(q-5)+\cdots+2=\sum_{k=0}^{m-1}(q-2-3k) = (q-2)m - \frac{3m(m-1)}{2} = \frac{q^2-q}{6}.\] 
This completes the proof.
\end{proof}

We now treat the remaining rational places of $Y_3$.

\begin{theorem}
\label{thm:rationalcase:other:semigroups}
Let $P$ be an $\mathbb{F}_{q^2}$-rational place of $Y_3$ not contained in $\mathcal{O}_0\cup \mathcal{O}_\infty$, such that $\alpha(P)^2-\alpha(P)+1 \neq 0$. Further, let $i$ be the $\PP$-order of $\alpha(P)$. If $i \le m-1$, then
\begin{equation*}
    H(P) = \langle q, q+1, (q-1) +j(q-2),(q-1)+i(q-2)-1 \mid j=0,\ldots, i-1\rangle.
\end{equation*}
If $i \geq m$, then
\begin{equation*}
    H(P) = \langle q, q+1, (q-1) +j(q-2) \mid j=0,\ldots, m-1\rangle.
\end{equation*}
\end{theorem}
\begin{proof}
The case $i\geq m$ can be treated nearly identically as in the proof of Theorem \ref{thm:special:short:orbit}. Instead of considering the functions $g_j/F_P^{j+1}$, one considers the functions $f_j/F_P^{j+1}$ for $j=0,\ldots, m-1$, where the $f_j$ are the functions constructed in Theorem \ref{thm:fi}. Hence, we are left to treat the case $i\leq m-1$.

We proceed as in the proof of Theorem \ref{thm:special:short:orbit}, showing that the semigroup $H:=\langle q-1, q, q+1, (q-1) +j(q-2), (q-1)+i(q-2)-1 \mid j=1,\ldots, i-1\rangle$ is contained in $H(P)$ and has at most $g(Y_3)$ gaps. We already know that $q,q+1 \in H(P)$. Let $F_{P}$ be as before and, for all $j$ such that $j=1,\ldots, i-1$, define the function
\begin{equation*}
    F_j:= \frac{f_j}{F_P^{j+1}},
\end{equation*}
where the $f_j$ are the functions constructed in Theorem \ref{thm:fi}.
The divisor of the function $F_j$ can be seen to be
\begin{equation*}
    (F_j) = E_j - ((q-1) + j(q-2))P
\end{equation*}
where $E_j\in \mathrm{Div}(Y_3)$ is an effective divisor such that $P\not \in \mathrm{supp}(E_j)$.
Therefore, $(q-1) +j(q-2)\in H(P)$ for all $j=1,\ldots, i-1$ and, similarly,
\[    (F_i) = E_i - ((q-1) + i(q-2)-1)P,\]
where $E_{i}\in \mathrm{Div}(Y_3)$ is an effective divisor such that $P\not \in \mathrm{supp}(E_{i})$.
This shows that $H\subseteq H(P)$.

In order to complete the proof, we need to show that the genus of the semigroup $H$ is less than or equal to $g(Y_3)$. 
We know $0 \in H$ and, as in the proof of Theorem \ref{thm:special:short:orbit}, we conclude that all integers in the set $\{j(q-2)+1,\dots,j(q+1)\}$ are in $H$ for any $j=1,\dots,i$. In addition, we have already proved that $(i+1)(q-2) \in H$, and adding $q-1$, $q$, and $q+1$ to the integers in $\{i(q-2)+1,\dots,i(q+1)\}$ gives that $\{(i+1)(q-2)+2,\dots,(i+1)(q+1)\} \subseteq H$.
Since $P$ is a rational place not in ${\mathcal O}_{\infty} \cup {\mathcal O}_{0}$, Lemma \ref{lem:q+1stpower} implies that $i+1$ divides $q+1$. We claim that for $k=0,\dots,\frac{q+1}{i+1}-1$ and all $j=1,\dots,i$ the sets $\{(k(i+1)+j)(q-2)+1,\dots,(k(i+1)+j)(q+1)\}$ are contained in $H$ as well as the integer $((k+1)(i+1))(q-2)$ and the set $\{(k+1)(i+1)(q-2)+2,\dots,(k+1)(i+1)(q+1)\}$. We have so far shown this for $k=0$. To complete the proof of the claim is then sufficient to notice that, if it is true for some $k-1 < \frac{q+1}{i+1}-1$, then adding $(i+1)(q-2)$ and the integers in $\{(i+1)(q-2)+2,\dots,(i+1)(q+1)\}$, shows that it is true for $k$ as well.

To count the number of gaps of $H$, let now $\ell:=k(i+1)+j-1$ and note that, for any $k$, we find precisely $q - 2 - 3\ell$ gaps $g$ such that $(k(i+1) + (j-1))(q+1) + 1 \leq g \leq (k(i+1) + j)(q-2)$. Since $q - 2 - 3\ell$ is non-negative if and only if $\ell \leq \left\lfloor \frac{q-2}{3} \right\rfloor = \frac{q-4}{3} = m-1$, as in the proof of Theorem \ref{thm:special:short:orbit}, we obtain then that
\begin{equation*}
    g(H) \leq \sum_{\ell=0}^{m-1}(q-2-3\ell) = (q-2)m - \frac{3m(m-1)}{2} = \frac{q^2-q}{6},
\end{equation*}
which concludes the proof.
\end{proof}

\begin{remark}
Since $2(q^2-q)/6-1=(q-1)+(m-1)q$ and $i \ge 1$, we see that the semigroups from Theorems \ref{thm:special:short:orbit} and \ref{thm:rationalcase:other:semigroups} are not symmetric. Combining this with Remark \ref{rem:nonsymm1}, we see that for $q>4$, the semigroup $H(P)$ is not symmetric for any of the rational places of $Y_3$. This stands in contrast to the situation for the otherwise quite similar maximal function field $X_3$ studied in \cite{BMV}. Indeed, in \cite[Remark 4.7]{BMV}, it was observed that $H(P)$ is symmetric for all rational places $P$ of $X_3$.  
\end{remark}

\section{Weierstrass semigroups at the non-\texorpdfstring{$\mathbb{F}_{q^2}$}{F\_q2}-rational places of \texorpdfstring{$Y_3$}{Y\_3}}
\label{sec:notrational}

In this subsection, we assume that $P=P_{(a,b)}$ is a place of $\Yte$ that does not come from a rational place of $Y_3$. With slight abuse of notation, we sometimes call such places the non-$\mathbb{F}_{q^2}$-rational places of $\Yte$. For such places, we have in particular that $\alpha(P)^2-\alpha(P)+1 \ne 0$ by Lemma \ref{lem:case1rationalpoints}. Our aim is now to compute the Weierstrass semigroup $H(P)$ by determining the gapsequence $G(P)$. As before, we denote by $Q=Q_{(A,B)}$ a place of $\He$ lying above $P$. Then $e(Q|P)=1$, so that functions in $\Yte$ can be written as a power series in the local parameter $T=(v-B)/B$.

Whereas we needed to construct special functions $f_j$ and $\tilde{f}_j$ from Theorems \ref{thm:fi} and \ref{thm:gi} when computing the Weierstrass semigroups of the rational places of $Y_3$, it turns out that another family of functions is needed for the computation of the semigroups of the places of $\Yte$ that do not come from the rational places of $Y_3$. This time there will be no need for two constructions, since we now only consider places $P$ of $\Yte$ that do not come from rational places of $Y_3$. In particular, we have $\alpha(P)^2-\alpha(P)+1 \neq 0$ for all places $P$ under consideration in this section. Another difference is that the $\PP$-order from Definition \ref{def:Pord} needs to be replaced by a slightly different concept, which we define now:

\begin{definition}\label{def:Qorder}
Let $\alpha \in \overline{\F}_{q^2} \setminus \{0,1,-\zeta_3,-\zeta_3^2\}$. Then we define the $\QQ$-order of $\alpha$ as the smallest positive integer $K$ such that $\QQ_{K+1}(1-\alpha)=0$, if such $K$ exists, and $\infty$ otherwise. If $P$ is a place of $\Yte$ such that $\alpha(P) \in \overline{\F}_{q^2} \setminus \{0,1,-\zeta_3,-\zeta_3^2\}$, then we define the $\QQ$-order of $P$ to be the $\QQ$-order of $\alpha(P).$ 
\end{definition}

\begin{remark}
    \label{rem:alphaeq2}
Since $\QQ_1(s)=s+1$, it seems perfectly possible that a place $P$ has $\QQ$-order zero. Indeed, if $\alpha(P)=2$, then $\QQ_1(1-\alpha(P))=0$. However, since $2^q-2^{q-1}=1$, Lemma \ref{lem:case1rationalpoints} implies that such a place would be $\mathbb{F}_{q^2}$-rational, contrary to our assumption in this section. This is the reason why we excluded this possibility in Definition \ref{def:Qorder}. 
\end{remark}

There is a strong connection between the $\PP$-order and the $\QQ$-order of a place $P$. We make this more precise in the following and start with a lemma.

\begin{lemma}\label{lem:Qord range}
Let $P$ be a place of $\Yte$ with $\PP$-order $i$. If $P$ has finite $\QQ$-order $\elltilde$, then $1\le\elltilde \le i+1$. Moreover, if an integer $\ell$ satisfies  $1 \le \ell \le i+1$ and $\left(\frac{\alpha(P)+\zeta_3}{\alpha(P)+\zeta_3^2} \right)^{3\ell+2}=\zeta_3$, then $\ell=\elltilde$.
\end{lemma}
\begin{proof}
We will write $\alpha:=\alpha(P)$ and $\beta:=(\alpha+\zeta_3)/(\alpha+\zeta_3^2)$ for convenience. By definition of $\PP$-order and $\QQ$-order we have $\beta^{3(i+1)}=1$ and $\beta^{3\elltilde+2}=\zeta_3$. Now write $3\ell+2=a\cdot 3(i+1)+r$, where $1 \le r \le i+1$. Then $r \equiv 2 \pmod 3$ and $\beta^r=1^{-a} \cdot \zeta_3=\zeta_3$. Since $\elltilde$ by definition is the smallest positive integer such that $\beta^{3\elltilde+2}=\zeta_3$, we may conclude that $a=0$.

If $\beta^{3\ell+2}=\zeta_3$, then $\beta^{3(\ell-\elltilde)}=1$. Since $\beta^3$ has multiplicative order $i+1$, a fact that follows directly from Definition \ref{def:Pord}, this implies that $i+1$ divides $\ell-\elltilde$. The lemma now follows. 
\end{proof}

We are now ready to indicate the relation between $\PP$-order and $\QQ$-order of a place.

\begin{proposition}\label{prop:relation_P_Q_order}
Let $i \in \mathbb{Z}$ be a positive integer such that $\gcd(i+1,q)=1$ and let $\varphi(\cdot)$ denote the Euler totient function. Then there exist exactly $(q^2-1)\varphi(i+1)$ many places $P$ of $\Yte$ with $\PP$-order $i$. If $i+1 \equiv 0 \pmod 3$, all of these have infinite $\QQ$-order.
If $i+1 \not\equiv 0 \pmod 3$, precisely $2(q^2-1)\varphi(i+1)/3$ of these places have an infinite $\QQ$-order, while the remaining $(q^2-1)\varphi(i+1)/3$ have a finite $\QQ$-order $\elltilde$, namely $\elltilde=2i/3$ if $i+1 \equiv 1 \pmod 3$ and $\elltilde=(i-1)/3$ if $i+1 \equiv 2 \pmod 3$.
\end{proposition}
\begin{proof}
As before, we will write $\alpha:=\alpha(P)$ and $\beta:=(\alpha+\zeta_3)/(\alpha+\zeta_3^2)$ in this proof for simplicity. First note that, by definition of the $\PP$-order, a place $P$ has $\PP$-order $i>0$ precisely if the multiplicative order of $\beta^3$ is equal to $i+1$. The case $i=0$ does not occur, since $\PP_1(s)=1$. This amounts to excluding the possibilities $\beta^3 \in \{0,1,\infty\}$ using Equation \eqref{def:alpha}.

In $\overline{\mathbb{F}}_{q^2}$ the multiplicative order of an element needs to be relatively prime to the characteristic, but otherwise there are no restrictions. Moreover, since any multiplicative subgroup of a finite field is cyclic and $\overline{\mathbb{F}}_{q^2}$ can be viewed as a union of all finite extensions of $\mathbb{F}_{q^2}$, there are for any $i>0$ with $\gcd(i+1,q)=1$, precisely $\varphi(i+1)$ many possibilities for $\beta^3$ in $\overline{\mathbb{F}}_{q^2}$. As a consequence, since the characteristic is not three, there are precisely $3\varphi(i+1)$ possibilities for $\beta$. Each $\beta$ corresponds to a unique value of $\alpha$ and, as we have observed before, for each $\alpha \not\in \{0,1,\infty\}$ there are precisely $(q^2-1)/3$ many places $P$ of $\Yte$ such that $\alpha(P)=\alpha$. Summing up, we conclude that for $i$ as above, there are precisely $(q^2-1)\varphi(i+1)$ places $P$ with $\PP$-order $i+1$.

Now if $\beta^3$ has multiplicative order $i+1$, then one has
\begin{equation}\label{eq:three possib}
\beta^{i+1}=1 \quad \text{or} \quad \beta^{i+1}=\zeta_3, \quad \text{or} \quad \beta^{i+1}=\zeta_3^2.
\end{equation}
Among the $(q^2-1)\varphi(i+1)$ places $P$ with $\PP$-order $i+1$, each of the three cases occurs equally often, since cubing is a three to one map from $\overline{\mathbb{F}}_{q^2}\setminus \{0\}$ to itself. Further, by direct computation one obtains that
\begin{equation}\label{eq:Qt}
\QQ_\ell(1-s)=(-1)^{\ell}\cdot \dfrac{\frac{1-\zeta_3}{3}(s+\zeta_3^2)^{3\ell-1}+\frac{1-\zeta_3^2}{3}(s+\zeta_3)^{3\ell-1}}{s},
\end{equation}
which implies that if one of the places $P$ we are considering has finite $\QQ$-order $\elltilde$, then 
\begin{equation}\label{eq:Qt2}
\beta^{3K+2}=\left( \frac{\alpha+\zeta_3}{\alpha+\zeta_3^2} \right)^{3\elltilde+2}=\zeta_3.
\end{equation}
In particular, we see that in that case, the multiplicative order of $\beta^3$ divides $3\elltilde+2$. One the order hand, we have already observed that the multiplicative order of $\beta^3$ is equal to $i+1$. We now distinguish several cases.

\textbf{Case 1.} If $i+1 \equiv 0 \pmod 3$ and $P$ has a finite $\QQ$-order $\elltilde$, then on the one hand $\beta^3$ has multiplicative order $i+1$, but on the other hand the multiplicative order of $\beta^3$ divides $3\elltilde+2$. But then three divides $3\elltilde+2$, a contradiction. Hence all $(q^2-1)\varphi(i+1)$ places with $\PP$-order $i+1$ have infinite $\QQ$-order.

\textbf{Case 2.} Assume $i+1 \equiv 1 \pmod 3$. If the first possibility in Equation \eqref{eq:three possib} holds and $P$ has finite $\QQ$-order $\elltilde$, then $\beta^{i+1}=1$ and $\beta^{3\elltilde+2}=\zeta_3$. Then the multiplicative order of $\beta$ divides $i+1$ on the one hand, but is a multiple of three on the other hand. Hence $\elltilde=\infty$. 

Now assume that the second possibility in Equation \eqref{eq:three possib} holds and $P$ has finite $\QQ$-order $\elltilde$. As in Case 1., we can conclude that the multiplicative order of $\beta$ is a multiple of three. On the other hand, we have $\beta^{i+1}=\zeta_3$ and $\beta^{3\elltilde+2}=\zeta_3$, which implies $\beta^{3\elltilde-i+1}=1$. Since $3\ell-i+1 \equiv 1 \pmod 3$, we see that the order of $\beta$ cannot be a multiple of three, a contradiction. Hence $\elltilde=\infty$.

Finally, assume that the third possibility in Equation \eqref{eq:three possib} holds. Then $\beta^{2i+2}=\zeta_3$ and since $2i+2 \equiv 2 \pmod 3$, Lemma \ref{lem:Qord range} implies that $\elltilde=2i/3$.

\textbf{Case 3.} Assume $i+1 \equiv 2 \pmod 3$. The first and third possibility in Equation \eqref{eq:three possib} can be handled very similarly as in Case 2. If $\beta^{i+1}=\zeta_3$, Lemma \ref{lem:Qord range} implies that $\elltilde=(i-1)/3$.
\end{proof}

\begin{remark}\label{rem:Q_ord_rational}
Proposition \ref{prop:relation_P_Q_order} and Lemma \ref{lem:q+1stpower} together imply that the rational places of $\Yte$ not in $\mathcal{O}_0 \cup \mathcal{O}_\infty$ and not satisfying $\alpha(P)^2-\alpha(P)+1 = 0$ are precisely those places with $\QQ$-order whose $\PP$-order $i$ satisfies that $i+1$ divides $q+1$.
\end{remark}

\begin{remark}\label{rem:from K to i}
Proposition \ref{prop:relation_P_Q_order} directly implies that if $P$ has finite $\QQ$-order $\elltilde$, then: 
\begin{enumerate}
    \item If $\elltilde$ is even, then $i=3\elltilde/2$ or $i=3\elltilde+1$.
    \item If $\elltilde$ is odd, then $i=3\elltilde+1$.
\end{enumerate}
Moreover, all these cases occur.
\end{remark}

We have now established all the facts that we will need concerning the $\QQ$-order of a place. Later on, we will also need a variant of the polynomials $\PP_j(s)$ and $\QQ_j(s)$ from Definition \ref{def:PQ}. The role of $\PP_j(s)$ will be played by $\QQ_\ell(1-s)$, while the role of $\QQ_j(s)$ will be played by a polynomial we define now:

\begin{definition}\label{def:R}
Let $\mathbb{F}$ be a field of characteristic different from three containing a primitive cube root of unity $\zeta_3$. Then for $i \in \mathbb{Z}$ we define the rational function
\begin{equation*}
\RR_i(s):=(-1)^i\frac{2\zeta_3+1}{3}\cdot \left((s+\zeta_3^2)^{3i-2}-(s+\zeta_3)^{3i-2})\right).
\end{equation*}
\end{definition}

From the definition, it is clear that $\RR_i(s)$ is a polynomial of degree at most $3i-3$ for $i \ge 1$. Moreover, in that case its constant term is equal to $(-1)^i$. In particular $\RR_1(s)=-1$. For later use, we prove an analogue of Lemma \ref{lem:PQ:identities} for these polynomials. The analogue will play a similar role as Lemma \ref{lem:PQ:identities} for the construction of certain functions.

\begin{lemma}
\label{lem:QtR:identities}
Let $i,j,k \in \Z$. Then
\begin{equation}
    \label{eq:id:3}
    \QQ_i(1-s)\QQ_{j+k}(1-s)-\QQ_j(1-s) \QQ_{i+k}(1-s) =9(-1)^{i+j+k}(s-1)^2(s^2-s+1)^{3j-1} \PP_{i-j}(s) \PP_{k}(s)
\end{equation}
and
\begin{equation}
    \label{eq:id:4}
    \QQ_i(1-s) \RR_{j+k}(s) -\QQ_j(1-s) \RR_{i+k}(s) =-3(-1)^{i+j+k}(s-1)^2(s^2-s+1)^{3j-1}\PP_{i-j}(s) \QQ_{k}(s).
\end{equation}
\end{lemma}
\begin{proof}
Using Equation \eqref{eq:Qt} and writing $S_1=s+\zeta_3$, $S_2=s+\zeta_3^2$, one obtains by direct computation
\begin{multline*}
(-1)^{i+j+k}s^2\QQ_i(1-s) \QQ_{j+k}(1-s) =\\
\frac{-\zeta_3}{3}{S_2^{3i+3j+3k-2}}+
\frac{1}{3}{S_1^{3j+3k-1}}S_2^{3i-1}
+\frac{1}{3}{S_1^{3i-1}S_2^{3j+3k-1}}
+\frac{-\zeta_3^2}{3}{S_1^{3i+3j+3k-2}}
\end{multline*}
and similarly
\begin{multline*}
(-1)^{i+j+k}s^2\QQ_j(1-s) \QQ_{i+k}(1-s) =\\
\frac{-\zeta_3}{3}{S_2^{3i+3j+3k-2}}+
\frac{1}{3}{S_1^{3i+3k-1}}S_2^{3j-1}
+\frac{1}{3}{S_1^{3j-1}S_2^{3i+3k-1}}
+\frac{-\zeta_3^2}{3}{S_1^{3i+3j+3k-2}}.
\end{multline*}
Hence
\begin{multline*}
(-1)^{i+j+k}s^2\left(\QQ_i(1-s) \QQ_{j+k}(1-s) -\QQ_j(1-s) \QQ_{i+k}(1-s)\right)=\\
\frac{(S_1S_2)^{3j-1}}{3}\left(S_1^{3k}S_2^{3(i-j)}+S_1^{3(i-j)}S_2^{3k}-S_1^{3(i-j+k)}-S_2^{3(i-j+k)}\right)\\
=-\frac{(S_1S_2)^{3j-1}}{3}(S_1^{3(i-j)}-S_2^{3(i-j)})(S_1^{3k}-S_2^{3k})\\
=9s^2(s-1)^2 (s^2-s+1)^{3j-1} \PP_{i-j}(s) \PP_{k}(s).
\end{multline*}
For the last equality, use Definition \ref{def:PQ} observing that $S_1S_2=s^2-s+1$ and $(\zeta_3-\zeta_3^2)^2=-3$.
Equation \eqref{eq:id:3} now follows. Equation \eqref{eq:id:4} can be proved very similarly.
\end{proof}

Now we show the existence of the functions we will need to compute the semigroup of the non-rational places.

\begin{proposition}
\label{prop:functions:gell}
Let $P$ be a place of $\Yte$ not coming from a rational place of $Y_3$. Further, let $\elltilde$ be the $\QQ$-order of $\al(P)$. If $\elltilde \le m-1$, then there exists a function $g_{\elltilde}\in L(K q P_{\infty}+2 D_0)$ such that $v_{P}(g_{\elltilde})=3\elltilde+2$. Moreover, for each $\ell\in \mathbb{Z}$ with $0\leq \ell \leq \min\{\elltilde-1,m-1\}$, there exists a function $g_\ell\in L(\ell qP_{\infty}+2 D_0)$ with $v_{P}(g_\ell)=3\ell+1$.
\end{proposition}

\begin{proof}
Let $P=P_{(a,b)}$ and write $\alpha=\alpha(P)$. Further, let $Q_{(A,B)}$ be a place of $\He$ lying over $P$ and write as before $T=(v-B)/B$. Let $\ell\leq \min\{K,m-1\}$. 

We start by showing the claim that there exists a function $g_{\ell}\in L(\ell q P_{\infty}+2 D_0)$ such that
\begin{equation*}
g_\ell=\frac{T^{3\ell+1}\cdot \left(\QQ_{\ell +1}(1-\alpha)-\RR_{\ell +1}(\alpha)T\right)}{(1+T)^{2}} +O(T^q).
\end{equation*}
We will show this claim using induction on $\ell$. If $\ell=0$, we define $$g_0=b(y_b-x_a)/y=b(y/b-x/a)/y.$$ First of all, note that $(y_b-x_a)(P_{(0,0)})=-1+1=0$, so that $v_{P_{(0,0)}}(y_b-x_a)\ge 1$. Moreover, using Equations \eqref{divx1} and \eqref{divy1}, we see that $v_{P_0^i}(y_b-x_a)=1$. Hence Equations \eqref{divx1} and \eqref{divy1} imply that $g_0 \in L(2 D_0)$. Further from Equations \eqref{eq:xapowerseries} and \eqref{eq:ybpowerseries}, we see that 
$$y/b=(1+T)^3$$ and 
$$y_b-x_a=(2-\alpha)T+(3-\alpha)T^2+T^3+O(T^q)=T(1+T)\left((2-\alpha)+T\right)+O(T^q).$$
Since $\QQ_1(1-\alpha)=2-\alpha$ and $\RR_1(\alpha)=-1$, the claim follows for $\ell=0$.

If $\ell=1$, we define $$g_1=3\left[3g_0^2\cdot y/b-(\alpha - 2)^2 \cdot f_0\right] - (\alpha^2 + 2\alpha - 2) \cdot f_0 \cdot g_0.$$ Using Equation \eqref{divy1} and the fact that $f_0 \in L(qP_\infty)$ and $g_0 \in L(2D_0)$, we can conclude that $g_1 \in  L(qP_\infty+2D_0).$
Moreover, from the above power series and using that $f_0=3T^2+(\alpha+1)T^3+O(T^q)$, a direct calculation shows that 
$$g_1=\frac{(\alpha^4 - 5\alpha^3 +
    10\alpha - 5)T^4-(4\alpha^3 - 6\alpha^2 + 1)T^5}{(1+T)^2}+O(T^q).$$
Since $\QQ_2(1-\alpha)=\alpha^4 - 5\alpha^3 + 10\alpha - 5$ and $\RR_2(\alpha) =4\alpha^3 - 6\alpha^2 + 1$, the claim follows for $\ell=1$.

If $\ell \ge 2$, we recursively define 
$$g_\ell:= \frac{\PP_2(\alpha) \cdot \QQ_{\ell-1}(1-\alpha) \cdot g_{\ell - 1} \cdot f_0 - \QQ_{\ell}(1-\alpha)\cdot g_{\ell - 2} \cdot f_1}{(\alpha^2-\alpha+1)^2 \QQ_{\ell-2}(1-\alpha)}.$$
Using the induction hypothesis for $\ell-1$ and $\ell-2$ as well as the power series for $f_0$ and $f_1$, we see that
$$g_\ell= \frac{T^{3\ell+1}(A(\alpha)-B(\alpha)T)}{(\alpha^2-\alpha+1)^2\QQ_{\ell-2}(1-\alpha)(1+T)^2},$$
where 
$$A(\alpha)=\QQ_{\ell-1}(1-\alpha)\QQ_{\ell}(1-\alpha)(\PP_2(\alpha)\QQ_1(\alpha)-\QQ_2(\alpha))+3\PP_2(\alpha)(\QQ_{\ell}(1-\alpha)\RR_{\ell-1}(\alpha)-\QQ_{\ell-1}(1-\alpha)\RR_\ell(\alpha))$$
and
$$B(\alpha)=\PP_2(\alpha)\QQ_{\ell-1}(1-\alpha)\RR_\ell(\alpha)\QQ_1(\alpha)-\QQ_{\ell}(1-\alpha)\RR_{\ell-1}(\alpha)\QQ_2(\alpha).$$
Now from Lemma \ref{lem:PQ:identities}, using that $\QQ_0(\alpha)=(\alpha^2-\alpha+1)^{-1}$, we obtain $\PP_2(\alpha)\QQ_1(\alpha)-\QQ_2(\alpha)=(\alpha^2-\alpha+1)^2$ and hence
$$A(\alpha)=\QQ_{\ell-1}(1-\alpha)\QQ_{\ell}(1-\alpha)(\alpha^2-\alpha+1)^2+3\PP_2(\alpha)(\QQ_{\ell}(1-\alpha)\RR_{\ell-1}(\alpha)-\QQ_{\ell-1}(1-\alpha)\RR_\ell(\alpha)).$$
Using Equation \eqref{eq:id:4} with $(i,j,k)=(\ell,\ell-1,0)$, we see that $\QQ_{\ell}(1-\alpha)\RR_{\ell-1}(\alpha)-\QQ_{\ell-1}(1-\alpha)\RR_\ell(\alpha)=3(\alpha^2-\alpha+1)^{3\ell-5}(\alpha-1)^2$. Therefore, we see that
$$A(\alpha)=(\alpha^2-\alpha+1)^2\left( \QQ_{\ell-1}(1-\alpha)\QQ_{\ell}(1-\alpha) +9 \PP_2(\alpha) (\alpha-1)^2 (\alpha^2-\alpha+1)^{3\ell-7})\right).$$
Applying Equation \eqref{eq:id:3} with $(i,j,k)=(\ell,\ell-2,1)$, we conclude that $$A(\alpha)=(\alpha^2-\alpha+1)^2 \QQ_{\ell-2}(1-\alpha)\QQ_{\ell+1}(1-\alpha),$$
which is exactly what we needed to show. 
Similarly one can show the required identity for $B(\alpha)$.
Indeed one has
$$\PP_2(\alpha)\QQ_{\ell-1}(1-\alpha)\RR_\ell(\alpha)\QQ_1(\alpha)=\QQ_{\ell-1}(1-\alpha)\RR_\ell(\alpha)\cdot(\QQ_2(\alpha)+(\alpha^2-\alpha+1)^2),$$
which implies that
$$B(\alpha)=\QQ_{\ell-1}(1-\alpha)\RR_\ell(\alpha)(\alpha^2-\alpha+1)^2+ (\QQ_{\ell-1}(1-\alpha)\RR_\ell(\alpha)-\QQ_{\ell}(1-\alpha)\RR_{\ell-1}(\alpha))\cdot\QQ_2(\alpha).$$
But then
$$B(\alpha)=\QQ_{\ell-1}(1-\alpha)\RR_\ell(\alpha)(\alpha^2-\alpha+1)^2- 3(\alpha^2-\alpha+1)^{3\ell-5}(\alpha-1)^2\cdot\QQ_2(\alpha),$$
which in turn implies that
$$B(\alpha)=(\alpha^2-\alpha+1)^2\QQ_{\ell-2}(1-\alpha) \RR_{\ell+1}(\alpha),$$
which is exactly what we needed to show.

It is at this point clear from the definition of $\QQ$-order that if $\ell \le \min\{\elltilde-1,m-1\}$, with $K$ the $\QQ$-order of $\alpha$, then $v_P(g_\ell)=3\ell+1.$ Similarly, it is clear that $v_P(g_\elltilde) \ge 3\elltilde+2$, if $\elltilde \le m-1$, since by definition of the $\QQ$-order, $\QQ_{\elltilde+1}(1-\alpha)=0.$ What remains to be shown is that in this case $\RR_{\elltilde+1}(\alpha) \neq 0.$ Equation \eqref{eq:Qt2} stated that if $\alpha$ has $\QQ$-order $\elltilde$, then 
$$\left( \frac{\alpha+\zeta_3}{\alpha+\zeta_3^2} \right)^{3\elltilde+2}=\zeta_3.$$
Similarly using Definition \ref{def:R}, $\RR_{\elltilde+1}(\alpha)=0$ precisely if 
$$\left( \frac{\alpha+\zeta_3}{\alpha+\zeta_3^2} \right)^{3\elltilde+1}=1.$$
If $\QQ_{\elltilde+1}(1-\alpha)=\RR_{\elltilde+1}(\alpha)=0,$ we therefore see that the multiplicative order of $(\alpha+\zeta_3)/(\alpha+\zeta_3^2)$ divides $\gcd(3(3\elltilde+2),3\elltilde+1)=1$. Therefore $(\alpha+\zeta_3)/(\alpha+\zeta_3^2)=1$, which is impossible.
\end{proof}

We now have all the needed ingredients to compute the gapsequence $G(P)$ for any place $P$ of $\Yte$ not coming from a rational place of $Y_3$. 

There will be several cases and we start with the most general one:

\begin{theorem}
    \label{thm:generic:semigroup}
Let $P$ be a non-$\mathbb{F}_{q^2}$-rational place of $\Yte$. Suppose that the $\QQ$-order of $\alpha(P)$ is at least $m$, then the set of gaps at $P$ is: 
\begin{equation*}
\label{eq:generic:gaps}
        G(P)=
        \{jq+k \mid j=0,\ldots,m-1,\ k=1,\ldots, q-2-3j\}.
\end{equation*}
\end{theorem}
\begin{proof}
First of all, note that the number of integers in the set $\{jq+k \mid j=0,\ldots,m-1,\ k=1,\ldots, q-2-3j\}$ equals:
$$\sum_{j=0}^{m-1}(q-2-3j)=m(q-2)-3\frac{(m-1)m}{2}=\frac{q^2-q}{6}=g(Y_3).$$
Since $G(P)$ contains exactly $g(Y_3)$ many gaps, the theorem follows if we can show that, for each $0 \le j \le m-1$ and $1 \le k \le q-2-3j$, the integer $jq+k$ indeed is a gap at $P$. Using Lemma \ref{lemma:diffdy}, we can prove this by constructing, for each $0 \le j \le m-1$ and $1 \le k \le q-2-3j$, a function $h_{j,k} \in L((m-1)(q+1)P_\infty+2D_0)$ such that $v_P(h_{j,k})=jq+k-1$. To construct these functions, we will use the functions $g_\ell \in L(\ell q P_\infty+2D_0)$, whose existence was shown in Proposition \ref{prop:functions:gell}. Note that the assumption that the $\QQ$-order of $\alpha(P)$ is at least $m$ implies that $v_P(g\ell)=3\ell+1$ for $\ell=0,\ldots,m-1$. We distinguish several cases:

\bigskip
\noindent{\bf Case 1:} For $0 \le j \le m-1$ arbitrary and $k=1$, we define the function 
$$h_{j,1}:=F_P^j,$$ 
with $F_P$ the function whose divisor is given in Equation \eqref{fundeq1}. Then $F_P \in L((q+1)P_\infty)$, while $v_P(F_P)=q$, since $P$ is not an $\mathbb{F}_{q^2}$-rational place of $\Yte.$ Denoting the pole divisor of the function $h_{j,1}$ by $(h_{j,1})_\infty$, we see that
$$(h_{j,1})_\infty \le j(q+1)P\infty \le (m-1)(q+1)P_\infty \le (m-1)(q+1)P_\infty+2D_0 \quad \text{and} \quad v_P(F_P^j)=jq,$$ 
just as needed.

\bigskip
\noindent{\bf Case 2:} For $0 \le j \le m-1$ arbitrary and $1< k < q-2-3j$, write $k=3\ell+\lambda$, where $\ell$ is a non-negative integer and $\lambda \in \{2,3,4\}$. The assumption $1<k < q-2-3j$, implies $0 \le \ell \le m-j-2$. 
Now define
\begin{equation*}
h_{j,k} := 
\begin{cases} 
F_{P}^j \cdot g_{\ell}\ &\mbox{if} \ \lambda=2,\\[5pt]
F_{P}^j \cdot g_{\ell} \cdot x_a \ &\mbox{if} \ \lambda =3,\\[5pt]
F_{P}^j \cdot g_{\ell} \cdot f_0 \ &\mbox{if} \ \lambda=4.\\
\end{cases}
\end{equation*}
Using Proposition \ref{prop:functions:gell} and using that $x_a,f_0 \in L(qP_\infty)$, we immediately see that the pole divisor of $h_{j,k}$ satisfies 
$$(h_{j,k})_\infty \le j(q+1)P_\infty+\ell q P_\infty+2D_0+ qP_\infty \le  (m-1)(q+1)P_\infty+2D_0.$$
In the final estimate, we used that $\ell \le m-j-2$ and $j \le m-1$, so that
$$j(q+1)+\ell q \le j(q+1)+(m-j-2)q=j+(m-2)q \le 1+(m-2)(q+1).$$
On the other hand, 
$$v_P(h_{j,k})=
\begin{cases} 
jq+3\ell+1     &\mbox{if} \ \lambda=2,\\[5pt]
jq+3\ell+1+1 \ &\mbox{if} \ \lambda =3,\\[5pt]
jq+3\ell+1+2 \ &\mbox{if} \ \lambda=4.\\
\end{cases}
$$
Hence in all cases $v_P(h_{j,k})=jq+k-1$, which is what we wanted to achieve.

\bigskip
\noindent{\bf Case 3:} For $0 \le j \le m-1$ arbitrary and $k=q-2-3j$, we define 
$$h_{j,q-2-3j}:=F_P^j g_{m-j-1}.$$ 
Then similarly as above,
$$(h_{j,q-2-3j})_\infty \le j(q+1)P_\infty+(m-j-1) q P_\infty+2D_0 \le  (m-1)(q+1)P_\infty+2D_0.$$
Further,
$$v_P(h_{j,q-2-3j})=jq+3(m-j-1)+1=jq+(q-2-3j)-1,$$
which is exactly what we we wanted to show.
\end{proof}

It is clear that only finitely many places can have a given $\QQ$-order. Because of this, the previous theorem describes the Weierstrass semigroup of all but finitely many places of $\Yte$. Therefore Theorem \ref{thm:generic:semigroup} describes the generic semigroup. For future convenience, we will write 
$$G_{\mathrm{gen}}:=\{jq+k \mid j=0,\ldots,m-1,\ k=1,\ldots, q-2-3j\},$$
for the gapset in Theorem \ref{thm:generic:semigroup}. A direct consequence of the proof of Theorem \ref{thm:generic:semigroup} is that all remaining non-$\mathbb{F}_{q^2}$-rational places of $\Yte$ have a different Weierstrass semigroup.
\begin{corollary}
Let $P$ be a non-$\mathbb{F}_{q^2}$-rational place of $\Yte$. Suppose that the $\QQ$-order of $P$ is at most $m-1$, then $G(P) \neq G_{gen}$.
\end{corollary}
\begin{proof}
Let $\elltilde$ be the $\QQ$-order of $P$ and assume that $\elltilde \le m-1$. Then $$v_P(F_P^{m-\elltilde-1}g_\elltilde)+1=(m-\elltilde-1)q+3\elltilde+3 \in G(P).$$ However, this integer was not in $G_{gen}.$
\end{proof}

This corollary implies that the Weierstrass places of $\Yte$, that is to say, all places such that $G(P) \neq G_{gen}$, are precisely the $\mathbb{F}_{q^2}$-rational places (since $q+1 \in G_{gen}$) and the non-$\mathbb{F}_{q^2}$-rational places with $\QQ$-order less than or equal to $m-1.$ What is left is to describe the gapset $G(P)$ of the remaining cases.

\begin{theorem}\label{thm:nonrational:weierstrass}
Let $P$ be a non-$\mathbb{F}_{q^2}$-rational place of $\Yte$. Let $i$ be the $\PP$-order and let $K$ be the $\QQ$-order of $P$, and suppose that $K \le m-1$.
Then
\begin{equation*}
\label{eq:nonrational:weierstrass}
\begin{split}
G(P)=(G_{gen} \setminus \{(m-\elltilde-1)q+3\elltilde+2-(q-3)\ell (i+1) \mid \ell=0,\ldots, \left\lfloor (m-K-1)/(i+1)\right\rfloor\} \\ \cup \{(m-\elltilde-1)q+3\elltilde+3-(q-3)\ell (i+1) \mid \ell=0,\ldots, \left\lfloor (m-K-1)/(i+1)\right\rfloor\}.
\end{split}
\end{equation*}
\end{theorem}
\begin{proof}
First of all, the cardinality of the putative gapset for $P$ is $g(Y_3)$. Indeed compared to $G_{gen}$ exactly $\left\lfloor (m-K-1)/(i+1)\right\rfloor+1$ gaps $\gamma$ are replaced by $\gamma+1$. 

Similarly as in the proof of Theorem \ref{thm:generic:semigroup}, we will indicate a function $\tilde{h}_{j,k} \in L((m-1)(q+1)P_\infty+2D_0)$ such that $v_P(\tilde{h}_{j,k})=jq+k-1$. If $0 \le j \le m-1$ and $1 \le k \le \min\{q-2-3j,3\elltilde+2\}$, we can simply put $\tilde{h}_{j,k}=h_{j,k}$, with $h_{j,k}$ as in the proof of Theorem \ref{thm:generic:semigroup}, since in that case $h_{j,k}$ was constructed using the functions $F_P$, $x_a$, $f_0$ and $g_0,\ldots,g_{K-1}$. Moreover, $v_P(\tilde{h}_{j,k})=v_P(h_{j,k})$ in these cases. 
If $k=3\elltilde+2$ and $0 \le j \le m-\elltilde-2$, we write $\tilde{h}_{j,3\elltilde+2}=F_P^j g_{\elltilde-1}f_0x_a$. Indeed this works, since 
$$v_P(\tilde{h}_{j,3\elltilde+2})+1=jq+3(\elltilde-1)+1+2+1+1=jq+3\elltilde+2$$
and, using $j \le m-\elltilde-2$,
$$(h_{j,3\elltilde+2})_\infty \le \left(j(q+1)+(\elltilde-1)q+q+\frac{2q+1}{3}\right)P_\infty+2D_0 \le (m-1)(q+1)P_\infty+2D_0.$$
Therefore we may from now on assume that $j=0,\ldots,m-\elltilde-1$ and $k \ge 3\elltilde+3$. 

We define the following quantities: 
$$\cc:=\left\lfloor \frac{k-3\elltilde-2}{3(i+1)} \right\rfloor, \quad \dd:=\left\lfloor \frac{k-3\elltilde-2}{3} \right\rfloor, \quad s:=\dd-\cc(i+1) \quad \text{and} \quad r:=(k-3\elltilde-2)-3\dd.$$
It is clear that $0 \le r \le 2$. One the other hand,
$$s=-\frac{r}{3}+\left( \frac{k-3\elltilde-2}{3(i+1)}-\left\lfloor \frac{k-3\elltilde-2}{3(i+1)} \right\rfloor \right)(i+1),$$ which immediately implies that $$-\frac{2}{3} \le s < i+1.$$ Since $s$ is an integer, we can conclude that $0 \le s \le i$. Note that we can write $k-3\elltilde-2=3(i+1)\cc+3s+r$. 
For all values of the putative gapset for $P$ it is true that $k \le q-1-3j$, but actually we will need to be a bit more precise: if $s=0$ and $r=0$, we have $k < q-2-3j$, while if $s=0$ and $r=1$ we have $k \le q-1-3j$. In all other cases, we have $k \le q-2-3j$.

In order to prove the theorem, we will construct functions of the form $\tilde{h}_{j,k}:=F_P^jg_\elltilde \cdot \hat{h}_{j,k}$ such that $v_P(\hat{h}_{j,k})=k-3\elltilde-3$ and $\hat{h}_{j,k} \in L(((m-j-1)(q+1)-\elltilde q)P_\infty)$. Indeed for such functions $\hat{h}_{j,k}$ it holds that $v_P(F_P^jg_\elltilde \cdot \hat{h}_{j,k})=jq+k-1$ and $F_P^jg_\elltilde \cdot \hat{h}_{j,k} \in L((m-1)(q+1)P_\infty+2D_0)$.

We now distinguish the following cases.
\begin{enumerate}
    \item If $s>0$, then we define:
\begin{equation*}
        \hat{h}_{j,k}:=\begin{cases}
            f_i^\cc \cdot f_{s-1} \quad &\mbox{if} \ r=0,\\[5pt]
            f_i^\cc \cdot f_{s-1} \cdot x_a \quad &\mbox{if} \ r=1,\\[5pt]
            f_i^\cc \cdot f_{s-1} \cdot f_0 \quad &\mbox{if} \ r=2.\\
        \end{cases}
    \end{equation*}
    \item If $s=0$, we define instead:
\begin{equation*}
        \hat{h}_{j,k}:=\begin{cases}
            f_i^{\cc-1}\cdot f_{i-1} \cdot f_0 \cdot x_a \quad &\mbox{if} \ r=0,\\[5pt]
            f_i^\cc \quad &\mbox{if} \ r=1,\\[5pt]
            f_i^\cc \cdot x_a \quad &\mbox{if} \ r=2.
        \end{cases}
    \end{equation*}
\end{enumerate}
We verify that these functions satisfy the right requirements one case at a time.
\begin{enumerate}
\item[1.a)]  $s>0$ and $r=0$. In this case $\dd=(k-3\elltilde-2)/3$. Moreover, $k \le q-2-3j$. We obtain 
$$v_P(f_i^\cc \cdot f_{s-1})=\cc (3i+3)+3(s-1)+2=3\dd-1=k-3\elltilde-3$$
and 
\begin{equation*}
\begin{split}
-v_{P_\infty}(f_i^\cc \cdot f_{s-1})\le & \cc (i+1)q+sq=\dd q = (k-3\elltilde-2)q/3 \\
 \le &  (q-2-3j-3\elltilde-2)q/3\le (m-1-j)(q+1)-\elltilde q.
\end{split}
\end{equation*}

\item[1.b)]  $s>0$ and $r=1$. In this case $\dd=(k-3\elltilde-3)/3$ and $k \le q-4-3j$, using that $k \le q-2-3j$ and the fact that $k \equiv 0 \pmod 3$ if $r=1$ (note indeed that $q-1 \equiv 0 \pmod 3$). 
Then similarly as before 
$$v_P(f_i^\cc \cdot f_{s-1} \cdot x_a)=\cc (3i+3)+3(s-1)+3=3\dd=k-3\elltilde-3$$
and
\begin{equation*}
\begin{split}
-v_{P_\infty}(f_i^\cc \cdot f_{s-1}\cdot x_a) \le & \cc (i+1)q+sq+q=(\dd+1) q = (k-3\elltilde)q/3 \\
 \le &  (q-4-3j-3\elltilde)q/3\le (m-1-j)(q+1)-\elltilde q.
\end{split}
\end{equation*}
    
\item[1.c)]  $s>0$ and $r=2$. In this case $\dd=(k-3\elltilde-4)/3$ and $k \le q-3-3j$ using $k \le q-2-3j$ and the fact that $k \equiv 1 \pmod 3$ if $r=2$.
We find
$$v_P(f_i^\cc \cdot f_{s-1} \cdot f_0)=\cc (3i+3)+3(s-1)+4=3\dd+1=k-3\elltilde-3$$
and
\begin{equation*}
\begin{split}
-v_{P_\infty}(f_i^\cc \cdot f_{s-1}\cdot f_0) \le & \cc (i+1)q+sq+q=(\dd+1) q = (k-3\elltilde-1)q/3 \\
 \le &  (q-3-3j-3\elltilde-1)q/3\le (m-1-j)(q+1)-\elltilde q.
\end{split}
\end{equation*}
\item[2.a)]  $s=0$ and $r=0$. In this setting we have $k<q-2-3j$ by definition of the putative gapset, and we also note that, since $k-3\elltilde-2=3(i+1)\cc+3s+r=3(i+1)\cc$, we cannot have $\cc=0$. Indeed, if $\cc=0$, then $k=3\elltilde+2$. However, we are assuming that $k \ge 3\elltilde+3$.

In this case $\dd=(k-3\elltilde-2)/3=\cc(i+1)$. Moreover, $k \le q-5-3j,$ using that in this particular case we have $k < q-2-3j$ and $r=0$. Hence 
$$v_P(f_i^{\cc-1} \cdot f_{i-1} \cdot f_0 \cdot x_a)=(\cc-1) (3i+3)+3i-1+2+1=\cc(3i+3)-1=3\dd-1=k-3\elltilde-3$$
and
\begin{equation*}
\begin{split}
-v_{P_\infty}(f_i^{\cc-1} \cdot f_{i-1} \cdot f_0 \cdot x_a) \le & ((\cc-1)(i+1)+i+1+1)q=\cc (i+1)q+q=(k-3\elltilde-2)q/3+q\\
\le & (q-5-3j-3\elltilde-2+3)q/3 \le (m-1-j)(q+1)-\elltilde q.
\end{split}
\end{equation*}

\item[2.b)]  $s=0$ and $r=1$. Note that in this case $k=q-1-3j$ is possible. 
We have that $\dd=(k-3\elltilde-3)/3=\cc(i+1)$ and $k \le q-1-3j$, hence 
$$v_P(f_i^{\cc})=\cc(3i+3)=3\dd=k-3\elltilde-3$$
and
\begin{equation*}
\begin{split}
-v_{P_\infty}(f_i^{\cc}) \le & \cc (i+1)q=\dd q=(k-3\elltilde-3)q/3\\
\le & (q-1-3j-3\elltilde-3)q/3 \le (m-1-j)(q+1)-\elltilde q.
\end{split}
\end{equation*}

\item[2.c)]  $s=0$ and $r=2$.
We have $\dd=(k-3\elltilde-4)/3=\cc(i+1)$ and $k \le q-3-3j.$ Hence 
$$v_P(f_i^\cc \cdot x_a)=\cc(3i+3)+1=3\dd+1=k-3\elltilde-3$$
and
\begin{equation*}
\begin{split}
-v_{P_\infty}(f_i^{\cc}\cdot x_a) \le & \cc (i+1)q+q=(\dd+1)q=(k-3\elltilde-1)q/3\\
\le & (q-3-3j-3\elltilde-1)q/3 \le (m-1-j)(q+1)-\elltilde q.
\end{split}
\end{equation*}
\end{enumerate}
Having checked all possible cases, we see that we indeed can construct all putative gaps. This finishes the proof.
\end{proof}

\begin{remark}
Proposition \ref{prop:relation_P_Q_order} and Remark \ref{rem:Q_ord_rational} together give a complete determination of all possible non-generic semigroups of the non-$\mathbb{F}_{q^2}$-rational places of $\Yte$. For the benefit of the reader, let us summarize the results. First of all, for any place $P$, its $\PP$-order $i$ needs to satisfy the condition $\gcd(i+1,q)=1$. Further, a place $P$ can only have a finite $\QQ$-order if $i+1 \not\equiv 0 \pmod 3$. To avoid the $\mathbb{F}_{q^2}$-rational places with finite $\QQ$-order, one needs to pose the additional condition that $i+1$ does not divide $q+1$. Then, given any positive integer $i$ satisfying these conditions, there are precisely $(q^2-1)\varphi(i+1)/3$ non-$\mathbb{F}_{q^2}$-rational places with finite $\QQ$-order $\elltilde$. More precisely, $\elltilde=2i/3$ if $i+1 \equiv 1 \pmod 3$ and $\elltilde=(i-1)/3$ if $i+1 \equiv 2 \pmod 3$. The semigroup described by such pairs $(i,\elltilde)$ is non-generic precisely if $\elltilde \le m-1$, so if $i+1 \equiv 1 \pmod 3$, one should consider $i \le 3(m-1)/2=(q-4)/2$, while if $i+1 \equiv 2 \pmod 3$, one should consider $i \le 3(m-1)+1=q-3.$
\end{remark}

\section*{Acknowledgments}
This work was supported by a research grant (VIL”52303”) from Villum Fonden. The third author was supported by an NWO Open Competition ENW – XL grant.

\end{document}